\documentclass[11pt,a4paper]{amsart}
\usepackage[english]{babel}

\newtheorem{theorem}{Theorem}[section]

\newtheorem{corollary}[theorem]{Corollary}
\newtheorem{lemma}[theorem]{Lemma}
\newtheorem{proposition}[theorem]{Proposition}

\newtheorem*{claim}{Claim}

\theoremstyle{definition}
\newtheorem*{definition}{Definition}

\newtheorem{remark}[theorem]{Remark}

\usepackage[square,sort,comma,numbers]{natbib}

\usepackage[colorlinks,
    linkcolor={red!50!black},
    citecolor={blue!50!black},
    urlcolor={blue!80!black}]{hyperref}

\usepackage{url}
\usepackage[all]{xy}
\usepackage{pstricks}
\usepackage{enumerate}
\usepackage{amsfonts,amssymb,amsmath,pinlabel,array,hhline}
\usepackage{slashed}
\usepackage{tabulary}
\usepackage{fancyhdr}
\usepackage{a4wide}
\usepackage{bbm,dsfont}
\usepackage[position=b]{subcaption}
\usepackage{graphicx}

\newcommand{\Z}{\mathds{Z}}
\newcommand{\Q}{\mathds{Q}}
\newcommand{\R}{\mathds{R}}

\newcommand{\id}{\operatorname{Id}}

\newcommand{\Hom}{\operatorname{Hom}}
\newcommand{\colim}{\operatorname{colim}}

\newcommand{\LZ}{\Z[t^{\pm 1}]}
\newcommand{\Bl}{\operatorname{Bl}}
\newcommand{\Ext}{\operatorname{Ext}}
\newcommand{\coker}{\operatorname{coker}}

\newcommand{\op}{\operatorname}

\newcommand{\bsm}{\left(\begin{smallmatrix}}
\newcommand{\esm}{\end{smallmatrix}\right)}

\newcommand{\sm}{\smallsetminus}
\newcommand{\bp}{\begin{pmatrix}}
\newcommand{\ep}{\end{pmatrix}}

\newcommand{\smfrac}[2]{\mbox{\footnotesize$\displaystyle\frac{#1}{#2}$}} 
\newcommand{\tmfrac}[2]{\mbox{\large$\frac{#1}{#2}$}} 
\newcommand{\wh}{\widehat}


\DeclareSymbolFont{EulerScript}{U}{eus}{m}{n}
\DeclareSymbolFontAlphabet\mathscr{EulerScript}
\begin{document}

\title{Characterisation of homotopy ribbon discs}

\author{Anthony Conway}
\address{Max Plank Institute for Mathematics, Bonn, Germany}
\email{anthonyyconway@gmail.com}
\author{Mark Powell}
\address{Department of Mathematical Sciences, Durham University, United Kingdom}
\email{mark.a.powell@durham.ac.uk}
\maketitle
\begin{abstract}
Let $\Gamma$ be either the infinite cyclic group $\Z$ or the Baumslag-Solitar group~$\Z \ltimes~\Z[\tmfrac{1}{2}]$. Let $K$ be a slice knot admitting a slice disc $D$ in the 4-ball whose exterior has fundamental group~$\Gamma$.  We classify the $\Gamma$-homotopy ribbon slice discs for $K$ up to topological ambient isotopy rel.\ boundary.
In the infinite cyclic case, there is a unique equivalence class of such slice discs.
When $\Gamma$ is the Baumslag-Solitar group, there are at most two equivalence classes of $\Gamma$-homotopy ribbon discs, and
at most one such slice disc for each lagrangian of the Blanchfield pairing of~$K$.
\end{abstract}

\section{Introduction}

A knot $K \subset S^3$ is \emph{slice} if it bounds a locally flat disc $D \subset D^4$.
The goal of this paper is to study the classification of the slice discs of a given slice knot up to topological ambient isotopy rel.\ boundary.
An initial observation is that one can connect sum a given slice disc with any 2-knot, to obtain infinitely many mutually non-isotopic slice discs for every slice knot, as can be seen by considering the fundamental group of the exterior.

We therefore restrict to slice discs~$D$ for which~$\pi_1(D^4 \setminus D)$ is a fixed group.
We also add a technical \emph{homotopy ribbon} condition on our discs by requiring that the inclusion map~$X_K:=S^3 \setminus \nu K \hookrightarrow N_D:= D^4 \setminus \nu D$ induces a surjection~$\pi_1(X_K) \twoheadrightarrow~\pi_1(N_D)$.
A knot is \emph{homotopy ribbon} if it admits such a \emph{homotopy ribbon disc.}
The (open) \emph{topological ribbon-slice conjecture} asserts that every slice knot is homotopy ribbon.

\begin{definition}
Given a group~$\Gamma$, a homotopy ribbon disc~$D$ is~\emph{$\Gamma$-homotopy ribbon if~$\pi_1(N_D)\cong~\Gamma$.}
An oriented knot is \emph{$\Gamma$-homotopy ribbon} if it bounds a~$\Gamma$-homotopy ribbon disc.
\end{definition}

We consider two cases:~the infinite cyclic group~$\Z$ and the Baumslag-Solitar group~\[G:= B(1,2) = \langle a,c \mid aca^{-1}=c^2 \rangle \cong \Z \ltimes \Z[\tfrac{1}{2}],\] where the generator $a$ of $\Z$ acts on $\Z[\tfrac{1}{2}]$ via multiplication by $2$.
Since both of these groups are solvable, and hence good in the sense of Freedman,
topological surgery in dimension~$4$ and the 5-dimensional $s$-cobordism theorem can be applied to classify $\Gamma$-homotopy ribbon discs.
A first question, however, is whether such discs exist.

The following theorem, whose two parts are respectively due to Freedman~\cite{FreedmanTopology} (see also~\cite[Theorem~11.7B]{FreedmanQuinn} and \cite[Appendix~A]{Garoufalidis-Teichner}) and Friedl-Teichner~\cite[Theorem~1.3]{FriedlTeichner} answers this question in the affirmative.
Let $M_K$ denote the zero-framed surgery manifold of~$K$.  Note that $\partial N_D = M_K$ for every slice disc $D$ for $K$.

\begin{theorem}\label{thm:existence-conditions}
Let~$K$ be an oriented knot.
\begin{enumerate}
\item If~$K$ has Alexander polynomial~$\Delta_K(t) \doteq 1$, then~$K$ is~$\Z$-homotopy ribbon.
\item If there is a surjection $\pi_1(M_K) \twoheadrightarrow G$ such that~$\operatorname{Ext}_{\Z[G]}^1(H_1(M_K;\Z[G]),\Z[G])=0$, then~$K$ is~$G$-homotopy ribbon.
\end{enumerate}
\end{theorem}

Since we now know that~$\Gamma$-homotopy ribbon discs exist for the groups~$\Gamma=\Z$ and~$G$, we return to our initial objective: their classification.

\subsection{$\Z$-homotopy ribbon discs}
In the~$\Z$ case, we show that the $\Z$-homotopy ribbon disc for an Alexander polynomial $1$ knot~$K$ is essentially unique. More precisely, we prove the following.

\begin{theorem}
\label{thm:IntroZTheorem}
Any two~$\Z$-homotopy ribbon discs for the same $\Z$-homotopy ribbon knot are ambiently isotopic rel.\ boundary.
\end{theorem}

Theorem~\ref{thm:IntroZTheorem} accords with Freedman's other famous result that every knotted~$S^2 \hookrightarrow S^4$ with~$\pi_1(S^4 \setminus S^2)=\Z$ is topologically isotopic to the standard unknotted embedding~$S^2 \hookrightarrow S^4$~\cite{FreedmanQuinn}.
We also note that Theorem~\ref{thm:IntroZTheorem} has recently been applied by Hayden in order to construct pairs of exotic ribbon discs~\cite{Hayden}.
We now move on to the~$\Z \ltimes \Z[\tmfrac{1}{2}]$ case.

\subsection{$\Z \ltimes \Z[\tmfrac{1}{2}]$-homotopy ribbon discs}
Before stating our second result, some additional notions are needed.
Recall that~$M_K$ denotes the~$0$-framed surgery along an oriented knot~$K$, that~$H_1(M_K;\Z[t^{\pm 1}])$ coincides with the Alexander module of~$K$ and that if~$D$ is a slice disc for~$K$, then~$\partial N_D=M_K$.
If~$D$ is a homotopy ribbon disc for a knot~$K$, then we call~\[P_D:=\ker( H_1(M_K;\Z[t^{\pm 1}]) \to H_1(N_D;\Z[t^{\pm 1}]))\] the \emph{lagrangian induced by~$D$}.
The reason for this terminology is that~$P_D$ is a lagrangian for the Blanchfield pairing~$\Bl(K)$ of~$K$, i.e.~$P_D=P_D^\perp$.
Note that if $K$ is merely slice, then this only need hold over the PID $\Q[t^{\pm 1}]$.

Our second main result expresses the classification of~$\Z \ltimes \Z[\tmfrac{1}{2}]$-homotopy ribbon discs using the induced lagrangians of the Blanchfield form.

\begin{theorem}
\label{thm:Intro}
Set~$G:=\Z \ltimes \Z[\tmfrac{1}{2}]$ and let~$K$ be a $G$-homotopy ribbon knot.
If two~$G$-homotopy ribbon discs for~$K$ induce the same lagrangian, then they are ambiently isotopic rel.\ boundary.
\end{theorem}

Before describing applications of Theorem~\ref{thm:Intro}, we outline the common strategy behind the proofs of Theorems~\ref{thm:IntroZTheorem} and~\ref{thm:Intro}.

We say that two slice discs~$D_1$ and~$D_2$ for a slice knot~$K$ are \emph{compatible} if there is an isomorphism~$f \colon \pi_1(N_{D_1}) \xrightarrow{\cong} \pi_1(N_{D_2})$ that satisfies~$f~\circ~\iota_{D_1}=\iota_{D_2}$, where~$\iota_{D_k} \colon \pi_1(M_K) \to \pi_1(N_{D_k})$ denotes the inclusion induced map for~$k=1,2$.
Observe that two~$\Z$-homotopy ribbon discs for an oriented $\Z$-homotopy ribbon knot are necessarily compatible, while Proposition~\ref{prop:ReformulateCompatibility} shows that $G$-homotopy ribbon discs are compatible if and only if they induce the same lagrangian.

Theorems~\ref{thm:IntroZTheorem} and~\ref{thm:Intro} are both consequences of the following result.

\begin{theorem}
\label{thm:Goal}
Use~$\Gamma$ to denote either~$\Z$ or~$\Z \ltimes \Z[\tmfrac{1}{2}]$ and let~$K$ be a $\Gamma$-homotopy ribbon knot.
 If~$D_1$ and~$D_2$ are two compatible $\Gamma$-homotopy ribbon discs for~$K$, then~$D_1$ and~$D_2$ are ambiently isotopic rel.\ boundary.
\end{theorem}

Theorem~\ref{thm:Goal} is proved by applying the surgery programme to the disc exteriors~$N_{D_1}$ and~$N_{D_2}$.
We briefly recall the steps of this well known classification programme.
Let~$D_1$ and~$D_2$ be two compatible~$\Gamma$-homotopy ribbon discs.
\begin{enumerate}
\item  In Lemma~\ref{lem:BG}, we establish that~$N_{D_1}$ and~$N_{D_2}$ are homotopy equivalent. In fact, they are aspherical and both $K(\Gamma,1)$ spaces.
\item Fixing a homotopy equivalence~$f \colon N_{D_1} \to N_{D_2}$, Proposition~\ref{prop:ReducedNormal} constructs a cobordism~$(W,N_{D_1},N_{D_2})$  relative to~$M_K$, and a degree one normal map
$$(F,\id_{N_{D_1}},f) \colon (W,N_{D_1},N_{D_2}) \to (N_{D_1} \times [0,1],N_{D_1},N_{D_1}).~$$
This is a surgery problem: we wish to know whether~$F$ is normally bordant to a (simple) homotopy equivalence. There is an obstruction~$\sigma(F)$ in the (simple) quadratic L-group~$L_5(\Z[\Gamma])$ to solving this problem.
\item After analysing the surgery obstruction~$\sigma(F)$ in Lemma~\ref{lem:Lgroup}, we take connected sums along circles with Freedman's~$E_8$ manifold times $S^1$, in order to replace~$F$ by a new degree one normal map with vanishing surgery obstruction.
\item  We perform 5-dimensional surgery to obtain an $s$-cobordism. Since~$\Gamma$ is a good group, the topological $s$-cobordism theorem in dimension~$5$ implies that~$N_{D_1}$ and~$N_{D_2}$ are homeomorphic rel.\ boundary.
\item Lemma~\ref{lem:Prelim} shows if the disc exteriors~$N_{D_1}$ and~$N_{D_2}$ are homeomorphic rel.\ boundary, then the discs~$D_1$ and~$D_2$ are ambiently isotopic rel.\ boundary.
\end{enumerate}

%
%
%

\subsection{Characterisation of homotopy ribbon discs}\label{subsection:characterisation-intro}

Theorems~\ref{thm:existence-conditions}~(1) and \ref{thm:IntroZTheorem}, combined with the fact that every knot with a $\Z$-homotopy ribbon disc has Alexander polynomial 1, yield the following characterisation.

\begin{theorem}\label{thm:characterisation-Z-intro}
  A knot $K$ has $\Delta_K(t) \doteq 1$ if and only if $K$ has a $\Z$-homotopy ribbon disc, unique up to ambient isotopy rel.\ boundary.
\end{theorem}

Now set $G:=\Z \ltimes \Z[\tmfrac{1}{2}]$.
In Section~\ref{section:characterising}, we shall combine Theorem~\ref{thm:Intro} with~\cite[Theorem~1.3]{FriedlTeichner} and further analysis to completely characterise $G$-homotopy ribbon discs.
To state our characterisation, we introduce some notation.
Given a $\Z[t^{\pm 1}]$-module $P$, we write $\overline{P}$ for~$P$ with the~$\Z[t^{\pm 1}]$-module structure induced by $t \cdot x=t^{-1}x$.
Note that $\Z[\tmfrac{1}{2}]$ is isomorphic as an abelian group to both $\Z[t^{\pm 1}]/(t-2)$ and~$\Z[t^{\pm 1}]/(2t-1)$, but the action of $t$ in the~$\Z[t^{\pm 1}]$-module structure differs -- either multiplication by $2$ or $\tmfrac{1}{2}$ respectively.

Let $P \subseteq H_1(M_K;\Z[t^{\pm 1}])$ be a submodule of $H_1(M_K;\Z[t^{\pm 1}])$ which is isomorphic to either~$\Z[t^{\pm 1}]/(t-2)$ or $\Z[t^{\pm 1}]/(2t-1)$, and such that $H_1(M_K;\Z[t^{\pm 1}])/P \cong~\overline{P}$.
In particular,~$\overline{P}$ is again isomorphic to $\Z[\tmfrac{1}{2}]$ for one of the module structures.
Associated with this submodule and a choice of meridian of the knot is a homomorphism
\begin{align*}
\phi_P \colon \pi_1(M_K) &\twoheadrightarrow \pi_1(M_K)/\pi_1(M_K)^{(2)} \cong \Z \ltimes \pi_1(M_K)^{(1)}/\pi_1(M_K)^{(2)}  \cong \Z \ltimes H_1(M_K;\Z[t^{\pm 1}]) \\
&\twoheadrightarrow \Z \ltimes H_1(M_K;\Z[t^{\pm 1}])/P \cong G
\end{align*}
which is obtained via canonical projections and the identification $\pi_1(M_K)^{(1)}/\pi_1(M_K)^{(2)} \cong H_1(M_K;\Z[t^{\pm 1}])$.
We can now state the complete algebraic characterisation of $G$-homotopy ribbon discs.  Details are given in Section~\ref{section:characterising}.

\begin{theorem}\label{thm:characterisation-Intro}
Set $G:=\Z \ltimes \Z[\tmfrac{1}{2}]$.
Let $K$ be an oriented knot, and let~$\mathcal{L}$ be the set of submodules $P \subseteq H_1(M_K;\LZ)$ of the Alexander module that are isomorphic to one of~$\Z[t^{\pm 1}]/(t-2)$ or~$\Z[t^{\pm 1}]/(2t-1)$ and fit into a short exact sequence
 \begin{equation}
 \label{eq:AwesomeAssumption0Intro}
0 \to P \to H_1(M_K;\Z[t^{\pm 1}]) \to \overline{P} \to 0.
\end{equation}
Mapping a $G$-homotopy ribbon disc to its induced lagrangian gives rise to a bijection between
\begin{itemize}
\item $G$-homotopy ribbon discs for $K$, up to topological ambient isotopy rel.\ boundary;
\item submodules $P \in \mathcal{L}$ such that, with respect to $\phi_P$,
\begin{equation*}
\tag{Ext} \label{eq:ExtConditionIntro}
\operatorname{Ext}_{\Z[G]}^1(H_1(M_K;\Z[G]),\Z[G])=0.
\end{equation*}
\end{itemize}
Moreover, these sets have cardinality at most two.
\end{theorem}

Note that Theorem~\ref{thm:characterisation-Intro} yields necessary and sufficient conditions for a knot to be $G$-homotopy ribbon.
This strengthens~\cite[Theorem 1.3]{FriedlTeichner}, which was stated in Theorem~\ref{thm:existence-conditions}.
\begin{corollary}
Set $G:=\Z \ltimes \Z[\tmfrac{1}{2}]$.
An oriented knot $K$ is $G$-homotopy ribbon if and only if its Alexander module contains a submodule $P$ that satisfies the following conditions:
\begin{enumerate}
\item $P\in \mathcal{L}$, that is $P$ is isomorphic to~$\Z[t^{\pm 1}]/(t-2)$ or~$\Z[t^{\pm 1}]/(2t-1)$, and we have $H_1(M_K;\Z[t^{\pm 1}])/P=~\overline{P}$,
\item $\operatorname{Ext}_{\Z[G]}^1(H_1(M_K;\Z[G]),\Z[G])=0$ with respect to $\phi_P$.
\end{enumerate}
\end{corollary}

The next remark notes that the situation in Theorem~\ref{thm:characterisation-Intro} can be made even more explicit.
\begin{remark}
\label{rem:Explicit}
As noted in Lemma~\ref{lem:ExtensionComputation}, the fact that $H_1(M_K;\LZ)$ fits into the short exact sequence~\eqref{eq:AwesomeAssumption0Intro} for some $P \in \mathcal{L}$ implies that $H_1(M_K;\LZ)$ must be isomorphic to one of
 \begin{equation*}
  M_1=\Z[t^{\pm 1}]/(t-2)(2t-1) \text{ or } M_2=\Z[t^{\pm 1}]/(t-2) \oplus \Z[t^{\pm 1}]/(2t-1).
  \end{equation*}
This strengthens the observation, due to Friedl and Teichner, that if a knot $K$ bounds a $G$-homotopy ribbon disc $D$, then~$\Delta_K \doteq (t-2)(2t-1)$~\cite[Corollary 3.4]{FriedlTeichner}.

For these $\LZ$-modules, Lemma~\ref{lem:TwoSubmodules} describes the set $\mathcal{L}$ of Theorem~\ref{thm:characterisation-Intro} explicitly:
\begin{itemize}
\item for $M_1$, we have $\mathcal{L}=\lbrace (t-2)M_1,(2t-1)M_1 \rbrace$;
\item for $M_2$, we have $\mathcal{L}=\lbrace  \Z[t^{\pm 1}]/(t-2) \oplus \lbrace 0 \rbrace, \lbrace 0 \rbrace \oplus \Z[t^{\pm 1}]/(2t-1) \rbrace$.
\end{itemize}
Finally, note that Theorem~\ref{thm:characterisation-Intro} ensures that if $K$ is $G$-homotopy ribbon, then both $P$ and $\overline{P}$ are lagrangians of the Blanchfield pairing $\Bl(K)$.
\end{remark}

\subsection{Examples}

After providing the proofs for these results, we shall describe an explicit application of Theorem~\ref{thm:Intro}: we study  the~$(\Z \ltimes \Z[\tmfrac{1}{2}])$-homotopy ribbon discs for the family~$\lbrace K_{n} \rbrace_{n \in \Z}$ of knots depicted in Figure~\ref{fig:KnIntro}. We recall the construction of explicit~$(\Z \ltimes \Z[\tmfrac{1}{2}])$-homotopy discs for each~$K_{n}$.
Then for $n=3k$, we obtain the following complete classification as an application of Theorem~\ref{thm:Intro}.

\begin{theorem}
\label{thm:ExampleIntro}
Set $G:=\Z \ltimes \Z[\tmfrac{1}{2}]$.
Up to ambient isotopy rel.\ boundary, the knot~$K_{3k}$ admits
\begin{enumerate}
\item precisely two distinct~$G$-homotopy ribbon discs if~$k=0,-1$;
\item a unique~$G$-homotopy ribbon disc if~$k \neq 0,-1.$
\end{enumerate}
\end{theorem}

\begin{figure}[!htb]
\centering
\labellist
    \small
    \pinlabel {$n$} at 92 373
    \pinlabel {$1$} at 683 283
    \endlabellist
\includegraphics[scale=0.24]{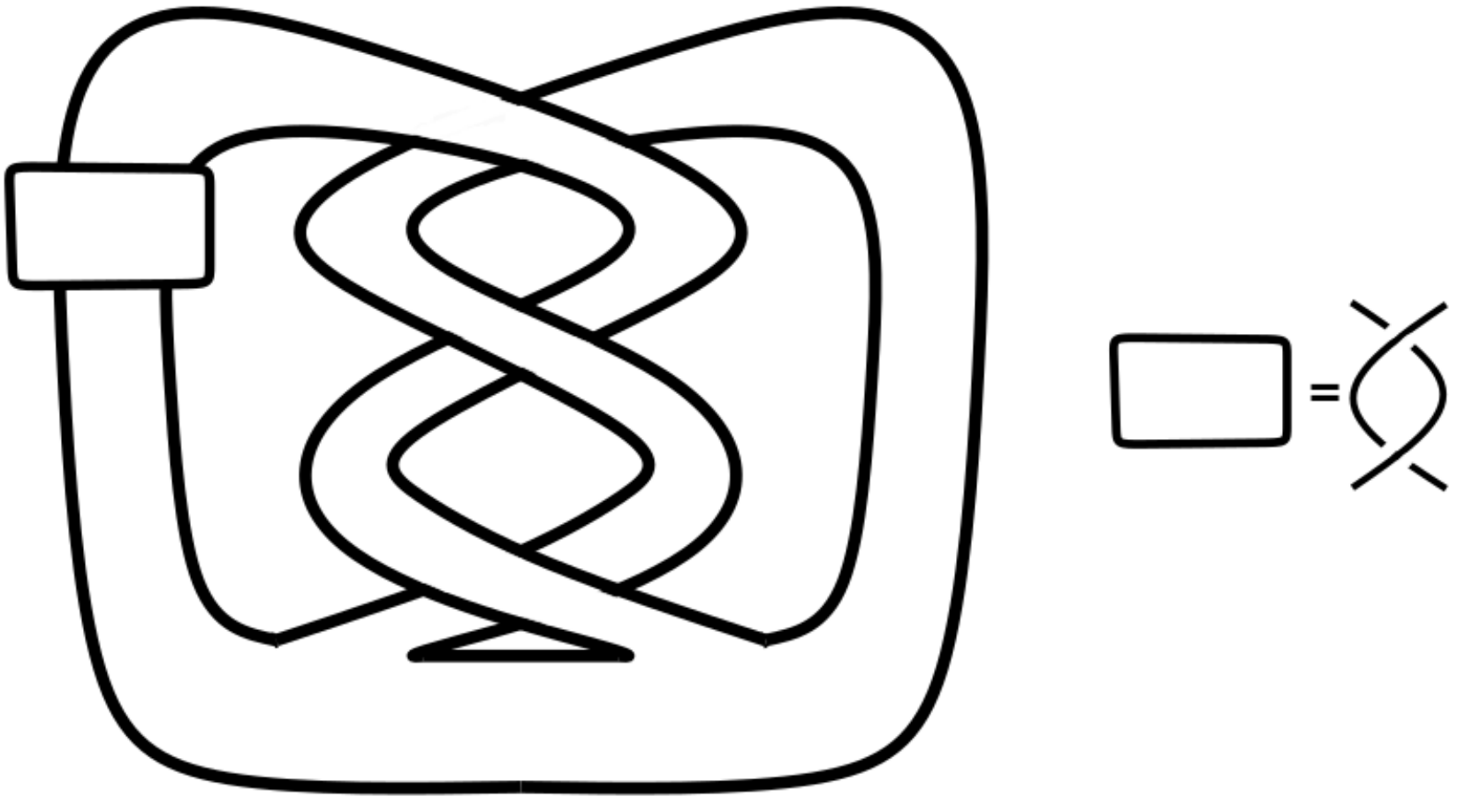}
\caption{The knot~$K_{n}$, where for $n>0$ the box symbolises $n$ positive full twists, as depicted on the right. For $n<0$, we use $|n|$ negative full twists.}
\label{fig:KnIntro}
\end{figure}

Since the Ext condition is difficult to verify in practice, the proof of the second item
uses  a theorem of Cochran-Harvey-Leidy~\cite{CochranHarveyLeidyDerivative} to obstruct the existence of a potential slice disc corresponding to one of the lagrangians of the Blanchfield pairing. This involves obtaining bounds on the Levine-Tristram signatures of metabolizing curves on a Seifert surface for $K_n$, as we shall explain in Section~\ref{sec:Examples}.
For $n \equiv 1,2 \mod{3}$, we have the following partial answer.  Part~(\ref{item:part-2-prop-other-n}) was obtained using a computer to calculate Levine-Tristram signatures.

\begin{proposition}\label{prop:other-n-intro}
Set $G:=\Z \ltimes \Z[\tmfrac{1}{2}]$. Up to ambient isotopy rel.\ boundary,
  \begin{enumerate}
    \item the knots $K_{-1}$ and $K_{-2}$ admit precisely two distinct~$G$-homotopy ribbon discs;
     \item\label{item:part-2-prop-other-n} the knots $K_{-5}$, $K_{-4}$, $K_1$, and $K_2$ admit a unique~$G$-homotopy ribbon disc.
  \end{enumerate}
\end{proposition}

 As $|n|$ increases, so does the complexity of the metabolizing curves for $K_n$.
 We therefore conjecture that $K_n$ admits precisely two $G$-homotopy ribbon discs for $-3 \leq n \leq 0$, and a unique $G$-homotopy ribbon disc otherwise. Note that altogether we have verified the conjecture for $-6 \leq n \leq 3$, and for $n=3k$, $k \in \Z$.
Due to limitations in our ability to obtain bounds for Levine-Tristram signatures of metabolizing curves in infinite families, we only have the experimental evidence given in Proposition~\ref{prop:other-n-intro}.

\subsection*{Organisation}

This article is organised as follows.
Theorem~\ref{thm:Goal} (and thus Theorem~\ref{thm:IntroZTheorem}) is proved using surgery theory in Section~\ref{sec:Thm12}, while we deduce Theorem~\ref{thm:Intro} from considerations on the Alexander module in Section~\ref{sec:Thm11}.
Theorem~\ref{thm:characterisation-Intro} is proved in Section~\ref{section:characterising},
Theorem~\ref{thm:ExampleIntro} and Proposition~\ref{prop:other-n-intro} are proved in Section~\ref{sec:Examples}.
Finally, in Section~\ref{sec:relaxing} we relax the rel.\ boundary condition on ambient isotopies, but still exhibit knots with precisely two $G$-homotopy ribbon discs.

\subsection*{Acknowledgments.}
AC thanks Durham University for its hospitality and was supported by an early Postdoc.Mobility fellowship funded by the Swiss FNS. He also thanks the mathematical research institute MATRIX in Australia where part of this research was performed. 
MP was partially supported by EPSRC New Investigator grant EP/T028335/1 and EPSRC New Horizons grant EP/V04821X/1.

Both authors thank Peter Feller, Fabian Hebestreit, Min Hoon Kim, Markus Land, Paolo Lisca, Allison N.\ Miller, Matthias Nagel, and Peter Teichner for helpful discussions and suggestions.
In particular Teichner's suggestions for Section~\ref{section:characterising} improved the statements of the results therein, and Section~\ref{sub:DegreeOne} benefited from discussions with Hebestreit and Land.
Part of our motivation to work on this problem came from an article of Juh\'{a}sz-Zemke~\cite{JuhaszZemkeDistinguishing}.

\subsection*{Conventions.}
Throughout this article, we work in the topological category and we assume that all manifolds are compact and oriented.
We say that homeomorphisms, homotopy equivalences and isotopies are \emph{rel.\ boundary} if they fix the boundary pointwise.
If~$N_1,N_2$ are two~$n$-manifolds with boundary~$M$, a cobordism between~$N_1$ and~$N_2$ is \emph{relative~$M$} if, when restricted to~$M$, it is the product~$M \times [0,1]$.
Given a Poincar\'e  complex~$(X,\partial X)$, a degree one normal map~$(f,\partial f) \colon (N,\partial N) \to (X,\partial X)$ is \emph{relative} if~$\partial f$ is a homotopy equivalence.

\section{The surgery programme for slice disc exteriors.}
\label{sec:Thm12}

In this section, we prove Theorem~\ref{thm:Goal} by following the surgery programme described above.
From now on, $\Gamma$ denotes either $\Z$ or $\Z \ltimes \Z[\tmfrac{1}{2}]$.
Recall that two $\Gamma$-homotopy ribbon discs~$D_1$ and~$D_2$ for a knot~$K$ are called \emph{compatible} if there is an isomorphism~$f \colon \pi_1(N_{D_1}) \stackrel{\cong}{\to} \pi_1(N_{D_2})$ that satisfies~$f~\circ~\iota_{D_1}=\iota_{D_2}$, where~$\iota_{D_k} \colon \pi_1(M_K) \to \pi_1(N_{D_k})$ denotes the inclusion induced map and $\partial N_{D_k} = M_K$ for~$k=1,2$.
Such an isomorphism~$f$ will be called a \emph{compatible isomorphism}.

\subsection{The homotopy type}
\label{sub:HomotopyEquivalent}

Let $D_1$ and $D_2$ be two $\Gamma$-homotopy ribbon discs for a knot $K$.
The first step in the surgery programme consists of showing that~$N_{D_1}$ and~$N_{D_2}$ have the same homotopy type.
To achieve this, we describe the homotopy type of arbitrary~$\Gamma$-homotopy ribbon disc exteriors: they are Eilenberg-Maclane spaces~$K(\Gamma,1)$.

\begin{lemma}
\label{lem:BG}
Let $\Gamma$ be either $\Z$ or $\Z \ltimes \Z[\tmfrac{1}{2}]$.
If~$D$ is a $\Gamma$-homotopy ribbon disc for a knot~$K$, then its exterior~$N_D$ is a~$K(\Gamma,1)$.
In particular,
\begin{enumerate}
\item all~$\Gamma$-homotopy ribbon disc exteriors are homotopy equivalent to one another;
\item two~$\Gamma$-homotopy ribbon discs are compatible if and only if they are homotopy equivalent rel.\ boundary.
\end{enumerate}
\end{lemma}
\begin{proof}
We must show that the higher homotopy groups of~$N_D$ vanish.
Since~$\pi_1(N_D)\cong \Gamma$, the~$\Gamma$-cover of~$N_D$ is simply connected.
Thus, by the Hurewicz theorem, we are reduced to showing that~$H_i(N_D;\Z[\Gamma])=~0$ for~$i \geq 2$.
We start with the case where~$i=3,4$.
Since~$K$ is homotopy ribbon, the map~$\pi_1(M_K) \to \pi_1(N_D)=\Gamma$ is surjective.
It follows that the corresponding~$\Gamma$-cover of $M_K$ is connected, so that we have an isomorphism $H_0(M_K;\Z[\Gamma]) \cong H_0(N_D;\Z[\Gamma])$. Therefore~$H_0(N_D,M_K;\Z[\Gamma])=0$.
Next, again since~$\pi_1(N_D)\cong \Gamma$, we have~$H_1(N_D;\Z[\Gamma])=0$, and we promptly deduce that~$H_1(N_D,M_K;\Z[\Gamma])=0$.
Poincar\'e duality and the universal coefficient spectral sequence, UCSS for short~\cite[Theorem~2.3]{Levine77}
 \[E_2^{p,q} = \overline{\Ext^q_{\Z[\Gamma]}(H_p(N_D,M_K;\Z[\Gamma]),\Z[\Gamma])} \Rightarrow H^{p+q}(N_D,M_K;\Z[\Gamma])\]
 imply that~$H_i(N_D;\Z[\Gamma]) \cong H^{4-i}(N_D,M_K;\Z[\Gamma])=0$ for~$i=3,4$. Here the overline emphasises the involuted module structure.
For~$i=2$, by duality and the UCSS (where we use that $H_i(N_D,M_K;\Z[\Gamma])=0$ for $i=0,1$), we have
$$H_2(N_D;\Z[\Gamma]) \cong H^2(N_D,M_K;\Z[\Gamma]) \cong \overline{\operatorname{Hom}_{\Z[\Gamma]}(H_2(N_D,M_K;\Z[G]),\Z[\Gamma])}.$$
It is therefore enough to show that~$H_2(N_D,M_K;\Z[\Gamma])$ is~$\Z[\Gamma]$-torsion.
Using the long exact sequence of~$(N_D,M_K)$ with $\Z[\Gamma]$ coefficients, this reduces to showing that~$H_2(N_D;\Z[\Gamma])$ and~$H_1(M_K;\Z[\Gamma])$ are both~$\Z[\Gamma]$-torsion.
The group~$\Gamma$ is PTFA since it is metabelian, has~$H_1(\Gamma)=\Z$ and torsion free commutator subgroup; we refer to~\cite[Definition~2.1 and Remark 2.3]{CochranOrrTeichner} for relevant details on PTFA groups.
 Since~$N_D$ is a~$\Z$-homology circle and since~$H_i(M_K;\Z)=H_i(S^1;\Z)$ for~$i=0,1$, these two statements follow from a now standard chain homotopy lifting argument~\cite[Proposition 2.10]{CochranOrrTeichner}.
We have therefore shown that~$N_D$ is a~$K(\Gamma,1)$.

The first consequence is immediate: for fixed~$\Gamma$ and~$n$, Eilenberg-Maclane spaces~$K(\Gamma,n)$ are unique up to homotopy equivalence.
We prove the last assertion.
 If~$f \colon N_{D_1} \to N_{D_2}~$ is a homotopy equivalence rel.\ boundary, then it certainly induces a compatible isomorphism~$\pi_1(N_{D_1}) \to \pi_1(N_{D_2})$.
Conversely, assume that~$f \colon \pi_1(N_{D_1}) \to \pi_1(N_{D_2})$ is a compatible isomorphism.
We use basic obstruction theory to construct the desired rel.\ boundary homotopy equivalence~$N_{D_1} \to N_{D_2}$.
Note that~$N_{D_i}$ is homotopy equivalent to a $3$-dimensional CW-complex with~$M_K$ as a subcomplex (an argument is provided in~\cite[Proof~of~Proposition~5.14]{ConwayNagelToffoli}).
We define a map~$N_{D_1}^{(1)} \cup M_K \to N_{D_2}$ by sending the (relative $M_K$) ~$1$-cells to their image under~$f$ and mapping~$M_K$ identically to its image in~$N_{D_2}$.
This map extends over the~$2$-cells of~$(N_{D_1},M_K)$: the attaching maps of the 2-cells are sent to the image of the relations under~$f$ and are therefore homotopically trivial.
Since we have established that the~$N_{D_i}$ are Eilenberg-Maclane spaces,~$\pi_2(N_{D_2})=0$ and~$\pi_3(N_{D_2})=0$, and we can therefore extend the aforementioned map over~$N_{D_1}$ as desired.
\end{proof}

\subsection{Finding a degree one normal map.}
\label{sub:DegreeOne}
Using Lemma~\ref{lem:BG}, we fix once and for all a rel.\ boundary homotopy equivalence~$f \colon N_{D_2} \to N_{D_1}$.
This way,~$\id_{N_{D_1}}$ and~$f$ are both degree one normal maps of the form~$(N_{D_j},\partial N_{D_j}=M_K) \to (N_{D_1},M_K)$,
and we wish to find a relative degree one normal cobordism~$W \to N_{D_1} \times [0,1]$ between them; we refer the reader to~\cite{WallSurgeryOnCompact} for the relevant terminology from surgery theory.
 In other words, we must show that~$f$ and~$\id_{N_{D_1}}$ define the same element in the set~$\mathcal{N}_{TOP}(N_{D_1},M_K)$ of relative normal bordism classes of degree one normal maps~$(M^4,\partial M^4) \to (N_{D_1},M_K)$.
To achieve this, we recall some facts from surgery theory that will be familiar to the experts.
\medbreak

Set $G:=\colim G(n)$ and $TOP:=\colim TOP(n)$,
where $G(n)$ and $TOP(n)$ denote respectively the monoid of homotopy self-equivalences of $S^{n-1}$ and the group of homeomorphisms of~$\R^n$ which map~$0$ to itself, both endowed with the compact-open topology.
We refer to~\cite{MadsenMilgram} for further details on $G$, $TOP$, and on the homotopy fibre $G/TOP$ of the map of classifying spaces $BTOP \to BG$.
Given a basepoint $*$ of~$G/TOP$ and a compact oriented topological~$4$-manifold $X$,
there are bijections
\begin{equation}
\label{eq:GTOP}
\mathcal{N}_{TOP}(X,\partial X) \simeq [(X,\partial X),(G/TOP,*)] \simeq H^4(X,\partial X;\Z) \oplus H^2(X,\partial X;\Z_2).
\end{equation}
Here, since~$X$ is a manifold,~$\mathcal{N}_{TOP}(X,\partial X)$ is based by~$\id_X$ and this leads to the first bijection in~\eqref{eq:GTOP}.  That the first map is an isomorphism uses topological map transversality~\cite[III.1]{KirbySiebenmann},~\cite[Section~9.5]{FreedmanQuinn}.
The second bijection follows from the fact that the Postnikov 4-type of~$G/TOP$ is homotopy equivalent to~$K(\Z,4) \times K(\Z_2,2)$;
see \cite[Annex~C, Remark 15.4]{KirbySiebenmann}, \cite[p.~397]{Kirby-Taylor}.

When~$X=N_{D_1}$, a combination of Poincar\'e duality and the universal coefficient theorem give~$H^2(N_{D_1},\partial N_{D_1};\Z_2)=~0$, starting from the fact that $N_{D_1}$ is a homology circle.
We therefore focus on the~$H^4$ term: composing the bijection of~\eqref{eq:GTOP} with the projection onto the first summand gives a map
\begin{equation}
\label{eq:proj1}
\op{proj}_1 \colon \mathcal{N}_{TOP}(X,\partial X) \to H^4(X,\partial X;\Z).
\end{equation}
Since~$H_3(X,\partial X;\Z) \cong H^1(X;\Z) \cong \Hom_\Z(H_1(X;\Z),\Z)$ is torsion free, we know that the evaluation map~$H^4(X,\partial X;\Z) \to \operatorname{Hom}_\Z(H_4(X,\partial X;\Z),\Z)$ is an isomorphism.
As~$X$ is compact, an element of~$H^4(X,\partial X;\Z)$ is determined by its evaluation on the fundamental class~$[X,\partial X]$.

\begin{proposition}
\label{prop:GTOP}
Let~$X$ be a compact oriented topological~$4$-manifold.
Given a degree one normal map~$(g,\partial g) \colon (M,\partial M) \to (X,\partial X)$ with~$\partial g$ a homotopy equivalence, one has
\[ \langle \op{proj}_1(g,\partial g),[X,\partial X] \rangle =\smfrac{1}{8} (\sigma(M)-\sigma(X)).\]
\end{proposition}

This result is known to surgery theorists.
We give a proof using \cite[Chapter~4]{MadsenMilgram}, but also refer to \cite[pp.~202-3]{FreedmanQuinn} for a related discussion.

\begin{proof}
As mentioned above, by~\cite[Annex C, Remark 15.4]{KirbySiebenmann} the map of $G/TOP$ to its fourth Postnikov section yields a $5$-equivalence
\[\Theta \colon G/TOP \to K(\Z,4) \times K(\Z_2,2). \]
Letting $k_4 \in H^4(K(\Z,4);\Z) \cong \Z$ and $k_2 \in H^2(K(\Z_2,2);\Z_2) \cong \Z_2$ be generators, this gives rise to cohomology classes
\begin{align*}
  h_4 &:= [(\op{pr}_1 \circ \Theta)^*(k_4)] \in H^4(G/TOP;\Z), \\
  h_2 &:= [(\op{pr}_2 \circ \Theta)^*(k_2)] \in H^2(G/TOP;\Z_2)
\end{align*}
where $\op{pr}_i$ is projection onto the $i$th factor.
The degree one normal map $(g,\partial g)$ determines $\wh{g} \in [(X,\partial X),(G/TOP,*)]$ by \eqref{eq:GTOP}. Then by definition of $\operatorname{proj}_1$, we have
\[\langle \op{proj}_1(g,\partial g),[X,\partial X] \rangle = \langle \wh{g}^*(h_4),[X,\partial X] \rangle.\]
Next, by \cite[Remark~4.36~and~p.~76]{MadsenMilgram}, we have:
\begin{equation}\label{eqn:MM-surgery-obstruction-map}
  \langle \wh{g}^*(h_4),[X,\partial X] \rangle = \smfrac{1}{8} (\sigma(M)-\sigma(X)).
\end{equation}
Madsen-Milgram give this formula for the class $\widetilde{K}_4 = h_4 \otimes 1 \in H^4(G/TOP;\Z) \otimes_{\Z} \Z_{(2)}$ $\cong H^4(G/TOP;\Z_{(2)})$, instead of our $h_4$, where
$\Z_{(2)}$ denotes the ring of integers localised at~$2$.  This is because they are describing the entire homotopy type of $G/TOP$. To describe the homotopy type succinctly,
as in Sullivan's study of $G/PL$~\cite[p.~126~onwards]{SullivanPhD}, one describes the homotopy type localised at~$2$, $G/TOP[2]$, and the homotopy type with $2$ inverted, and then combines them.  But as we are only interested in the $4$-type, the map~$\Theta$ describes the homotopy type without localising.   Note that the formula~\eqref{eqn:MM-surgery-obstruction-map} is the same whether we use~$h_4$ or $h_4 \otimes 1$, since $\Z \subseteq \Z_{(2)}$.
This concludes the proof of Proposition~\ref{prop:GTOP}.
\end{proof}

Using Proposition~\ref{prop:GTOP}, we can establish the existence of the desired normal bordism.

\begin{proposition}
\label{prop:ReducedNormal}
Let $\Gamma$ be either $\Z$ or $\Z \ltimes \Z[\tmfrac{1}{2}]$.
Let $D_1$ and $D_2$ be two $\Gamma$-homotopy ribbon discs for a knot $K$ and let $f \colon N_{D_1} \to N_{D_2}$ be a rel.\ boundary homotopy equivalence.
There exists a rel.\ $M_K$ cobordism~$(W,N_{D_1},N_{D_2})$ and a relative degree one normal map
$$(F,\id_{N_{D_1}},f) \colon (W,N_{D_1},N_{D_2}) \to (N_{D_1} \times [0,1],N_{D_1},N_{D_1}).~$$
\end{proposition}
\begin{proof}
We show that the degree one normal maps~$\id_{N_{D_1}}$ and~$f$ define the same class in the normal set~$\mathcal{N}_{TOP}(N_{D_1},M_K)$.
We already argued that~$H^2(N_{D_1},M_K;\Z_2)=0$, whence the fact that the map~$\operatorname{proj}_1 \colon \mathcal{N}_{TOP}(N_{D_1},M_K) \to H^4(N_{D_1},M_K;\Z)$ described in~\eqref{eq:proj1} is a bijection.
Proposition~\ref{prop:GTOP} now implies that~$\id_{N_{D_1}}$ and~$f$ define the same class in~$\mathcal{N}_{TOP}(N_{D_1},M_K)$: in both cases, we know that~$\tmfrac{1}{8}(\sigma(N_{D_i})-\sigma(N_{D_1}))$ vanishes, since $H_2(N_{D_i};\Z)=0$ for $i=1,2$. This concludes the proof Proposition~\ref{prop:ReducedNormal}.
\end{proof}

\subsection{The surgery obstruction.}
\label{sub:SurgeryObstruction}

Proposition~\ref{prop:ReducedNormal} gives rise to a~$5$-dimensional surgery problem.
This surgery problem has a surgery obstruction in~$L_5(\Z[\Gamma])$.
Here, since the Whitehead groups~$\operatorname{Wh}(\Z \ltimes \Z[\tmfrac{1}{2}])$ and $\operatorname{Wh}(\Z)$ are zero,  we omitted the decorations in the $L$-groups.
That the Whitehead group $\operatorname{Wh}(\Z \ltimes \Z[\tmfrac{1}{2}])$ vanishes is due to Waldhausen~\cite[Theorem~5]{Waldhausen-I}, since $\Z \ltimes \Z[\tmfrac{1}{2}]$ is a torsion-free one-relator group.  We also refer to~\cite[Lemma~6.4]{HambletonKreckTeichner} for a shorter explanation.
The next lemma describes~$L_5(\Z[\Gamma])$ for $\Gamma=\Z,\Z \ltimes \Z[\tmfrac{1}{2}]$.

\begin{lemma}
\label{lem:Lgroup}
For $\Gamma=\Z$ and $\Gamma= \Z \ltimes \Z[\tmfrac{1}{2}]$, there is an isomorphism~$L_5(\Z[\Gamma]) \cong L_4(\Z)$.
\end{lemma}
\begin{proof}
For $\Gamma=\Z$, this follows immediately from Shaneson splitting~\cite{ShanesonSplitting}, namely one has~$L_5(\Z[\Z])=L_4(\Z) \oplus L_5(\Z)=L_4(\Z)$.
We therefore focus on the case $G=\Z \ltimes \Z[\tmfrac{1}{2}]$.
Invoking the Shaneson splitting $L_4(\Z[\Z])=L_4(\Z)$, it is enough to show that
$$L_5(\Z[G]) \cong L_4(\Z[\Z]).$$
Multiplication by~$2$ induces an automorphism of~$\Z[\tmfrac{1}{2}]$.
Let~$\alpha_*$ be the induced automorphism of~$L_n(\Z[\tmfrac{1}{2}])$.
Using Ranicki's long exact sequence for twisted Laurent extensions~\cite{RanickiTwisted} (see also~\cite[Theorem 4.5]{FriedlTeichner}), we obtain the following exact sequence:
\begin{equation}
\label{eq:lesLgroups}
\xymatrix @C-0.3cm{
\cdots \ar[r] & L_5(\Z[\Z[\tfrac{1}{2}]]) \ar[r]^{1-\alpha_*}& L_5(\Z[\Z[\tfrac{1}{2}]]) \ar[r] & L_5(\Z[G]) \ar[r] & L_4(\Z[\Z[\tfrac{1}{2}]]) \ar[r]^{1-\alpha_*} & L_4(\Z[\Z[\tfrac{1}{2}]]) \ar[r] & \cdots
}
\end{equation}
As explained in~\cite[p. 2149]{FriedlTeichner}, one has an isomorphism~$L_4(\Z[\Z[\tfrac{1}{2}]])\cong L_4(\Z[\Z])$, and the induced map~$\alpha_* \colon L_4(\Z[\Z[\tfrac{1}{2}]]) \to L_4(\Z[\Z[\tfrac{1}{2}]])$ is the identity map.
Arguing as in~\cite[p. 2149]{FriedlTeichner}, one can use the fact that~$L$-groups commute with colimits (direct limits) to show that~
$L_5(\Z[\Z[\tfrac{1}{2}]])=0$ (in~\cite{FriedlTeichner}, the authors show that~$ L_3(\Z[\Z[\tfrac{1}{2}]])=0$, but the same argument applies here).
The lemma now follows from the exact sequence displayed in~\eqref{eq:lesLgroups}.
\end{proof}

\subsection{The proof of Theorem~\ref{thm:Goal}.}
\label{sub:ProofTheoremCompatible}
We are now in position to prove Theorem~\ref{thm:Goal}, which states that if~$D_1$ and~$D_2$ are two compatible homotopy~$\Gamma$-ribbon discs for~$K$ with $\Gamma=\Z,\Z \ltimes \Z[\tmfrac{1}{2}]$, then~$D_1$ and~$D_2$ are ambiently isotopic rel.\ boundary.

\begin{proof}[Proof of Theorem~\ref{thm:Goal}]
We first combine the results of the previous lemmas.
Since~$D_1$ and~$D_2$ are compatible, Lemma~\ref{lem:BG} ensures the existence of a homotopy equivalence~$f \colon N_{D_1} \to N_{D_2}$ rel.\ boundary.
 Proposition~\ref{prop:ReducedNormal} provides a relative degree one normal map
~$$(F,\id_{N_{D_1}},f) \colon (W,N_{D_1},N_{D_2}) \to (N_{D_1} \times [0,1],N_{D_1},N_{D_1}).$$
 The surgery obstruction~$\sigma(F)$ lies in~$L_5(\Z[\Gamma])$.
Lemma~\ref{lem:Lgroup} implies that~$L_5(\Z[\Gamma]) \cong L_4(\Z)$ and it is known that~$L_4(\Z)=L_0(\Z) \cong 8\Z$ is detected by the signature; see e.g.~\cite{HusemollerMilnor}.
As a consequence, we think of~$\sigma(F)$ as an integer.
Next, we modify $F$ to a new surgery problem~$F'$ with vanishing surgery obstruction.
This is achieved by connect summing~$W$ with~$\sigma(F)$ copies of the degree one normal map~$S^1 \times \pm E_8 \to S^1 \times S^4$.
As in~\cite[p. 206]{FreedmanQuinn}, this connect sum is performed along loops; the next paragraph provides some details on this construction.

First, we may assume that the degree one normal map~$F \colon W \to N_{D_1} \times [0,1]$ is a homeomorphism~$F^{-1}(N_{D_1} \times [0,\varepsilon]) \to N_{D_1} \times [0,\varepsilon]$ in a collar neighbourhood of $N_{D_1} \times [0,1]$.
Next, choose an embedded~$S^1 \times D^4 \subset N_{D_1} \times [0,\varepsilon]$ whose core represents a meridian of $D_1$, and consider its preimage~$F^{-1}(S^1 \times D^4) \subseteq W$.
The domain of our new map is obtained by replacing the domain of the map~$F^{-1}(S^1 \times D^4) \to S^1 \times D^4$ with the domain of the degree one map~$S^1 \times \operatorname{cl}(E_8 \setminus D^4) \to S^1 \times D^4$.
Our new degree one normal map $F'$ is obtained by modifying~$F$ using this map on the new~$S^1 \times \operatorname{cl}(E_8 \setminus D^4)$.

The outcome of this construction is a degree one normal map~$F' \colon W' \to (N_{D_1} \times [0,1])$ with vanishing surgery obstruction and which coincides with~$F$ on the boundary.
It follows that~$F'$ is normal bordant rel.\ $M_K \times [0,1]$ to a homotopy equivalence.
We deduce that~$N_{D_1}$ and~$N_{D_2}$ are~$s$-cobordant rel.\ boundary.
Since the group~$\Gamma$ is solvable (for $\Z$ this is immediate, while $G=\Z \ltimes \Z[\tmfrac{1}{2}]$ is metabelian i.e.~$G^{(2)}=1$), it is good in the sense of Freedman~\cite{FreedmanQuinn} (see also~\cite{FreedmanTeichner, KrushkalQuinn}).
The 5-dimensional $s$-cobordism theorem thus implies that~$N_{D_1}$ is homeomorphic to~$N_{D_2}$ rel.\ boundary~\cite[Theorem 7.1A]{FreedmanQuinn}.
Lemma~\ref{lem:Prelim} below shows that this homeomorphism gives rise to an ambient isotopy from~$D_1$ to~$D_2$.
\end{proof}

The next lemma concludes the proof of Theorem~\ref{thm:Goal}.

\begin{lemma}
\label{lem:Prelim}
Let~$D_1$ and~$D_2$ be slice discs for~$K$. The following assertions are equivalent:
\begin{enumerate}
\item the discs $D_1$ and~$D_2$ are ambiently isotopic rel.\ boundary;
\item the exteriors~$N_{D_1}$ and~$N_{D_2}$ are homeomorphic rel.\ boundary.
\end{enumerate}
\end{lemma}
\begin{proof}
Let~$(g_t \colon D^4 \to D^4)_{t \in [0,1]}$ be an ambient  isotopy rel.\ boundary from~$D_1$ to~$D_2$.
In other words, the~$g_t$ are homeomorphisms,~$g_0=\operatorname{id}_{D^4}$ and~$g_1 \colon D^4 \stackrel{\cong}{\to} D^4$ satisfies~$g_1(D_1)=~D_2$.
It follows that~$g_1$ induces a well defined rel.\ boundary homeomorphism~$N_{D_1} \to~N_{D_2}$.

Now to the converse.
Start from a rel.\ boundary homeomorphism~$f \colon N_{D_1} \to~N_{D_2}$.
We wish to attach~$2$-handles to~$N_{D_1}$ and~$N_{D_2}$ in order to recover a self-homeomorphism of~$D^4$.
Note that for~$i=1,2$, we have
$$ M_K=\partial N_{D_i}\cong \overline{S^3 \setminus (K \times D^2)} \cup (D_i \times \partial D^2).$$
As a consequence, we have an identification of~$D_1 \times \partial D^2$ with~$D_2 \times \partial D^2$.
Making use of this identification, we attach a two handle~$D^2 \times D^2$ to both~$N_{D_1}$ and~$N_{D_2}$ with core~$D_1 \times D^2=D_2 \times D^2$.
The resulting manifolds are homeomorphic to~$D^4$ and respectively contain~$D_1$ and~$D_2$ as slice discs for~$K$.
Since the homeomorphism~$f$ fixes~$M_K=\partial N_{D_1}$ pointwise, it extends to a well defined homeomorphism
$$f':=f \cup \id_{D^2 \times D^2} \colon D^4 \to D^4.$$
By construction, this homeomorphism carries~$D_1$ to~$D_2$.
Since~$f$ is equal to the identity on the boundary, so is~$f'$.
We can therefore apply Alexander's trick: this result implies that~$f'$ is isotopic rel.\ boundary to the identity homeomorphism.
We have therefore established that~$D_1$ and~$D_2$ are ambiently isotopic rel.\ boundary.
This concludes the proof of the lemma.
\end{proof}

\section{The proof of Theorem~\ref{thm:Intro}.}
\label{sec:Thm11}

From now on, we write $\Z[t^{\pm 1}]$ instead of $\Z[\Z]$ and recall that the lagrangian induced by a homotopy ribbon disc~$D$ is
\[P_D:=\ker(H_1(M_K;\Z[t^{\pm 1}]) \to H_1(N_D;\Z[t^{\pm 1}])).\]
Thanks to Theorem~\ref{thm:Goal}, in order to conclude the proof of Theorem~\ref{thm:Intro}, it remains to show that if two~$(\Z \ltimes \Z[\tmfrac{1}{2}])$-homotopy ribbon discs induce the same lagrangian of the Blanchfield pairing, then they are compatible.
In fact, in Proposition~\ref{prop:ReformulateCompatibility} below, we will show that these two conditions are equivalent.

\medbreak

First we show that if~$D$ is homotopy ribbon, then the Alexander module~$H_1(N_D;\Z[t^{\pm 1}])$ can be described as a quotient of the Alexander module~$H_1(X_K;\Z[t^{\pm 1}])$ by the lagrangian~$P_D$.

\begin{lemma}
\label{lem:AlexModule}
If~$D$ is a homotopy ribbon disc for a knot~$K$, then the inclusion~$\iota_D \colon X_K \hookrightarrow N_D$ induces a~$\Z[t^{\pm 1}]$-isomorphism
$$(\iota_D)_* \colon H_1(X_K;\Z[t^{\pm 1}])/P_D \stackrel{\cong}{\to} H_1(N_D;\Z[t^{\pm 1}]).$$
\end{lemma}
\begin{proof}
It is enough to show that~$\iota_D$ induces a surjection~$(\iota_D)_* \colon H_1(X_K;\Z[t^{\pm 1}]) \stackrel{\cong}{\to} H_1(N_D;\Z[t^{\pm 1}])$ between the Alexander modules.
Recall that these modules can be identified with derived quotients, namely
\[H_1(N_D;\Z[t^{\pm 1}]) \cong \pi_1(N_D)^{(1)}/\pi_1(N_D)^{(2)} \text{ and } H_1(X_K;\Z[t^{\pm 1}]) \cong \pi_1(X_K)^{(1)}/\pi_1(X_K)^{(2)}.\]
The lemma will therefore follow once we observe that~$\iota_{D}$ restricts to a surjection
$$ \iota_D \colon \pi_1(X_K)^{(1)} \to  \pi_1(N_D)^{(1)}.$$
Indeed: if~$ \iota_D$ is a surjection, then so is~$(\iota_D)_*$.
Next, we use the abelianisation homomorphisms~$\phi_K$ and~$\phi_D$ of~$\pi_1(X_K)$ and~$\pi_1(N_D)$.
 The inclusion~$\iota_D \colon X_K \hookrightarrow N_D$ induces an isomorphism~$H_1(X_K;\Z) \stackrel{\cong}{\to} H_1(N_D;\Z)$.
We also denote this map by~$\iota_D$ and observe that~$\iota_D \circ \phi_K=\phi_D \circ \iota_D$.
Furthermore, the kernels of~$\phi_K$ and~$\phi_D$ are isomorphic to the respective commutator subgroups:
\begin{align*}
&\pi_1(N_D)^{(1)}=\ker(\phi_D),\\
&\pi_1(X_K)^{(1)}=\ker(\phi_K).
\end{align*}
The lemma will thus be proved once we show that~$\iota_D$ induces a surjection
$\ker(\phi_K) \to~\ker(\phi_D)$.
Let~$y$ lie in~$\ker(\phi_D)$.
Since~$D$ is homotopy ribbon, the map
~$\iota_{D} \colon \pi_1(X_K) \to~\pi_1(N_D)$ is surjective and we can therefore choose an~$x \in \pi_1(X_K)$ such that~$\iota_D(x)=y$.
 Using the aforementioned equality~$\iota_D \circ \phi_K=\phi_D \circ \iota_D$, we deduce that~$\iota_D (\phi_K(x))=\phi_D(\iota_D(x)) = \phi_D(y)=0$.
 Since~$\iota_D$ is an isomorphism on homology, we obtain~$\phi_K(x)=0$, establishing that~$x$ lies in~$\ker(\phi_K)$.
This concludes the proof of the lemma.
\end{proof}

Next, we describe two consequences of
Lemma~\ref{lem:AlexModule}.
\begin{corollary}
\label{cor:FromModuleToGroup}
Let~$D$ be a homotopy ribbon disc for a knot~$K$.
\begin{enumerate}
\item The inclusion~$M_K \hookrightarrow N_D$ induces a~$\Z[t^{\pm 1}]$-isomorphism
$$(\iota_D)_* \colon H_1(M_K;\Z[t^{\pm 1}])/P_D \stackrel{\cong}{\to} H_1(N_D;\Z[t^{\pm 1}]).$$
\item Set~$G:=\Z \ltimes \Z[\tmfrac{1}{2}]$.
If~$D_1$ and~$D_2$ are~$G$-homotopy ribbon discs, then a $\Z[t^{\pm 1}]$-linear isomorphism~$f \colon H_1(N_{D_1};\Z[t^{\pm 1}]) \stackrel{\cong}{\to} H_1(N_{D_2};\Z[t^{\pm 1}])$ that satisfies $f \circ (\iota_{D_1})_* = (\iota_{D_2})_* $ gives rise to a compatible isomorphism~$\pi_1(N_{D_1}) \stackrel{\cong}{\to} \pi_1(N_{D_2})$.
\end{enumerate}
\end{corollary}
\begin{proof}
To prove the first assertion, combine the isomorphism~$H_1(X_K;\Z[t^{\pm 1}])=H_1(M_K;\Z[t^{\pm 1}])$ with Lemma~\ref{lem:AlexModule}.
Next, we prove the second assertion.
The groups~$ \pi_1(M_K)^{(1)}/\pi_1(M_K)^{(2)} = H_1(M_K;\Z[t^{\pm 1}])$ and~$\pi_1(M_K)/\pi_1(M_K)^{(1)}=H_1(M_K;\Z)$ fit into the following short exact sequence of groups:
$$ 1 \to H_1(M_K;\Z[t^{\pm 1}]) \to \pi_1(M_K)/\pi_1(M_K)^{(2)} \stackrel{p}{\to} H_1(M_K;\Z) \to 1. ~$$
Since~$H_1(M_K;\Z) \cong \Z$ is freely generated by a meridian of~$K$, if we fix a \emph{based} meridian for~$K$, then we get a splitting~$s$ of~$p$.
Thus, the map
\begin{align*}
\Z \ltimes H_1(M_K;\Z[t^{\pm 1}]) &\stackrel{\cong}{\to} \pi_1(M_K)/\pi_1(M_K)^{(2)} \\ (n,h) &\mapsto s(n)h
\end{align*} is an isomorphism.
Next, let~$D$ be a~$G$-homotopy ribbon disc for~$K$.
Since the inclusion~$M_K \hookrightarrow~N_D$ induces an isomorphism~$H_1(M_K;\Z) \stackrel{\cong}{\to} H_1(N_D;\Z)$, the choice of a based meridian for~$K$ also gives a splitting of~$ \pi_1(N_D)/\pi_1(N_D)^{(2)} \twoheadrightarrow H_1(N_D;\Z)$, and the same argument as above yields an isomorphism~$\Z \ltimes H_1(N_D;\Z[t^{\pm 1}]) \cong \pi_1(N_D)/\pi_1(N_D)^{(2)}$.
On the other hand, since the group~$\pi_1(N_{D}) \cong~G$ is metabelian (i.e.\ $G$ satisfies $G^{(2)}=1$), we have~$\pi_1(N_{D}) = \pi_1(N_{D})/\pi_1(N_{D})^{(2)}$.
Combining these facts, we deduce that
$$ \pi_1(N_{D}) = \pi_1(N_{D})/\pi_1(N_{D})^{(2)} \cong \Z \ltimes H_1(N_D;\Z[t^{\pm 1}]).$$
To conclude, let~$D_1$ and~$D_2$ be~$G$-homotopy ribbon discs for the knot~$K$, and fix a~$\Z[t^{\pm 1}]$-linear isomorphism~$f \colon H_1(N_{D_1};\Z[t^{\pm 1}]) \stackrel{\cong}{\to} H_1(N_{D_2};\Z[t^{\pm 1}]).$
The isomorphism~$\pi_1(N_{D_1})  \stackrel{\cong}{\to} \pi_1(N_{D_2})$ is constructed by combining~$f$ with the isomorphism~$\varphi \colon H_1(N_{D_1};\Z)=\Z \stackrel{\cong}{\to} \Z=H_1(N_{D_2};\Z)$ that maps a meridian of~$D_1$ to a meridian of~$D_2$.
More precisely, the aforementioned splitting~$s \colon H_1(M_K;\Z) \to \pi_1(M_K)/\pi_1(M_K)^{(2)}$ of~$p$ induces analogous splittings for~$N_{D_1}$ and~$N_{D_2}$ and this choice ensures that~$(\varphi,f)$ gives an isomorphism~$\Z \ltimes H_1(N_{D_1};\Z[t^{\pm1}]) \to \Z \ltimes H_1(N_{D_2};\Z[t^{\pm1}])$. The second assertion follows and the lemma is proved.
\end{proof}

The following proposition concludes the proof of Theorem~\ref{thm:Intro}.

\begin{proposition}
\label{prop:ReformulateCompatibility}
Set~$G:=\Z \ltimes \Z[\tmfrac{1}{2}]$. Two~$G$-homotopy ribbon discs~$D_1$ and~$D_2$ for a knot~$K$ induce the same lagrangian if and only if they are compatible.
\end{proposition}
\begin{proof}
As in Corollary~\ref{cor:FromModuleToGroup}, we use $(\iota_{D_j})_*$ to denote the inclusion induced maps on the level of the Alexander modules.
Assume that~$D_1$ and~$D_2$ are compatible and choose a compatible isomorphism~$f \colon \pi_1(N_{D_1}) \stackrel{\cong}{\to} \pi_1(N_{D_2})$.
Passing to the derived quotients, this isomorphism induces a $\Z[t^{\pm 1}]$-linear isomorphism $f_*$ that satisfies $f_* \circ (\iota_{D_1})_*=(\iota_{D_2})_*$.
We therefore obtain~$P_{D_1}=P_{D_2}$, as desired.

Conversely, assume that~$P_{D_1}=P_{D_2}$.
Using the first item of Corollary~\ref{cor:FromModuleToGroup}, we know that the inclusions induce
isomorphisms
$ (\iota_{D_j})_* \colon  H_1(M_K;\Z[t^{\pm 1}])/P_{D_j} \stackrel{\cong}{\to} H_1(N_{D_j};\Z[t^{\pm 1}])$ for~$j=1,2$.
Consequently, setting~$f_*:=(\iota_{D_2})_* \circ (\iota_{D_1})_*^{-1}$, we obtain a $\Z[t^{\pm 1}]$-linear isomorphism~$H_1(N_{D_1};\Z[t^{\pm 1}]) \stackrel{\cong}{\to} H_1(N_{D_2};\Z[t^{\pm 1}])$.
By construction, this isomorphism satisfies~$f_* \circ (\iota_{D_1})_*=(\iota_{D_2})_*$
Using the second item of Corollary~\ref{cor:FromModuleToGroup}, we can thus extend~$f_*$ to a compatible isomorphism~$\pi_1(N_{D_1}) \stackrel{\cong}{\to} \pi_1(N_{D_2})$.
This concludes the proof of the proposition.
\end{proof}

\section{Characterising $G$-homotopy ribbon discs.}\label{section:characterising}

In this section, as promised in Section~\ref{subsection:characterisation-intro}, we explain how our results combine with those of Friedl-Teichner~\cite{FriedlTeichner} to give a characterisation of $G$-homotopy ribbon discs. In particular, we prove Theorem~\ref{thm:characterisation-Intro} from the introduction.
Given a $\Z[t^{\pm 1}]$-module $P$, we use $\overline{P}$ to denote~$P$ with the~$\Z[t^{\pm 1}]$-module structure induced by $t \cdot x=t^{-1}x$.
Throughout this section, we also adopt the convention that $\Z[\tmfrac{1}{2}]$ denotes either~$\Z[t^{\pm 1}]/(t-2)$ or $\Z[t^{\pm 1}]/(2t-1)$, and that if~$\Z[\tmfrac{1}{2}]=\Z[t^{\pm 1}]/p(t)$ for~$p(t)=t-2$ or~$2t-1$, then $\overline{\Z[\tmfrac{1}{2}]}$ denotes $\Z[t^{\pm 1}]/p(t^{-1})$.
\medbreak

We start with some necessary conditions for a knot $K$ to bound a $G$-homotopy ribbon disc, some of which were touched on in~\cite{FriedlTeichner}.

\begin{proposition}\label{prop:necessary-condition}
Let $D$ be a $G$-homotopy ribbon disc for a knot $K$.
\begin{enumerate}
  \item The Alexander module of $K$ sits in a short exact sequence
  \begin{equation}
  \label{eq:NecessaryCondition}
0 \to P_D \to H_1(M_K;\Z[t^{\pm 1}]) \to \overline{P}_D  \to 0,
  \end{equation}
with the induced lagrangian $P_D$
 isomorphic to either $\Z[t^{\pm 1}]/(t-2)$ or $\Z[t^{\pm 1}]/(2t-1)$. In particular, $\Delta_K \doteq (t-2)(2t-1)$.
\item With respect to the inclusion induced map  $\phi \colon \pi_1(M_K) \twoheadrightarrow \pi_1(D^4 \setminus \nu D) \cong G$, one has
\begin{equation}
\label{eq:ExtCondition}
\operatorname{Ext}_{\Z[G]}^1(H_1(M_K;\Z[G]),\Z[G])=0.
\end{equation}
\end{enumerate}
\end{proposition}

\begin{proof}
Using Poincar\'e duality, the UCSS, and the fact that $D$ is homotopy ribbon, we see that $H_2(N_D;\Z[t^{\pm 1}])=0$.
Combining this with a glance at the long exact sequence of the pair~$(N_D,M_K)$ with $\Z[t^{\pm 1}]$ coefficients shows that
$$P_D=\operatorname{im}(H_2(N_D,M_K;\Z[t^{\pm 1}]) \to H_1(M_K;\Z[t^{\pm 1}])) \cong H_2(N_D,M_K;\Z[t^{\pm 1}]) .$$
Next, observe that $H_1(N_D;\Z[t^{\pm 1}]) \cong \overline{\Z[\tmfrac{1}{2}]}$ and $H_1(N_D,M_K;\Z[t^{\pm 1}])=0:$
for the absolute homology module, use that $H_1(N_D;\Z[t^{\pm 1}])=G^{(1)}/G^{(2)}=\overline{\Z[\tmfrac{1}{2}]}$ (we fix our choice of $p(t)$ in the convention from the start of the section so that this equation holds). For the relative homology module, use that $D$ is homotopy ribbon.
The long exact sequence of the pair~$(N_D,M_K)$ with~$\Z[t^{\pm 1}]$ coefficients now gives rise to the short exact sequence displayed in~\eqref{eq:NecessaryCondition}.

In order to conclude the proof of the first item, it remains to argue that~$P_D$ is isomorphic to $\Z[\tmfrac{1}{2}]$.
First, note that $\overline{\Hom(H_2(N_D;\Z[t^{\pm 1}]),\Z[t^{\pm 1}])}=0$: indeed, we argued that~$H_2(N_D;\Z[t^{\pm 1}])=0$) and~$\overline{\Ext^2_{\Z[t^{\pm 1}]}(\Z,\Z[t^{\pm 1}])}=0$ (using for instance group cohomology).
We then combine these facts with Poincar\'e duality, the UCSS, and the fact that~$H_1(N_D;\Z[t^{\pm 1}])=\overline{\Z[\tmfrac{1}{2}]}$, to complete the proof of the first item:
\begin{align*}
P_D &\cong H_2(N_D,M_K;\Z[t^{\pm 1}]) \cong H^2(N_D;\Z[t^{\pm 1}]) \cong \overline{\Ext^1_{\Z[t^{\pm 1}]}(H_1(N_D;\Z[t^{\pm 1}]),\Z[t^{\pm 1}])} \\ &\cong \overline{\Ext^1_{\Z[t^{\pm 1}]}(\overline{\Z[\tmfrac{1}{2}]},\Z[t^{\pm 1}])} \cong \Z[\tmfrac{1}{2}].
\end{align*}

Now we establish the second item of the proposition.
According to Friedl-Teichner~\cite[Lemma~5.1]{FriedlTeichner}, the Ext condition displayed in~\eqref{eq:ExtCondition} is equivalent to the vanishing of a $\Z[G]$ coefficient Blanchfield form $\Bl_G^K \colon H_1(M_K;\Z[G]) \times H_1(M_K;\Z[G]) \to Q(G)/\Z[G]$, where $Q(G)$ is the Ore localisation of  $\Z[G]$.
Using the arguments of~\cite[pages 461-462]{CochranOrrTeichner}, one can establish the existence of a Blanchfield-type pairing
\[\Bl_G^D \colon H_2(N_D,M_K;\Z[G]) \times H_1(N_D;\Z[G]) \to Q(G)/\Z[G].\]
Essentially, one uses that $H_*(N_D;Q(G))=0$, and argues that the appropriate Bockstein homomorphism is an isomorphism.
Using $A^{\wedge}$ to denote $\Hom_{\Z[G]}(A,Q(G)/\Z[G])$, the same arguments as in~\cite[pages 461-462]{CochranOrrTeichner} then show that the following diagram commutes:
\begin{equation}
\label{eq:DiagramUniqueExist}
\xymatrix{H_2(N_D,M_K;\Z[G]) \ar[r]^-{\partial} \ar[d]^{\Bl_G^D} & H_1(M_K;\Z[G]) \ar[r]^-{j} \ar[d]^{\Bl_G^K} & H_1(N_D;\Z[G])  \ar[d]^{\Bl_G^D} \\
  H_1(N_D;\Z[G])^{\wedge} \ar[r]^-{j^{\wedge}} & H_1(M_K;\Z[G])^{\wedge} \ar[r]^-{\partial^{\wedge}} & H_2(N_D,M_K;\Z[G])^{\wedge}.}
  \end{equation}
In~\eqref{eq:DiagramUniqueExist}, the vertical maps indicate the adjoints to the aforementioned Blanchfield pairings.
  Now~$\pi_1(N_D) \cong G$ implies that $H_1(N_D;\Z[G]) =0$.
A quick diagram chase then shows that~$\Bl_G^K=0$: given $x,y \in H_1(M_K;\Z[G])$, by exactness, and since $H_1(N_D;\Z[G]) =0$, there is an $u \in H_2(N_D,M_K;\Z[G])$ with $\partial u=x$; the commutativity of the diagram displayed in~\eqref{eq:DiagramUniqueExist} then gives
  \[\Bl_G^K(x)(y) = \Bl_G^K(\partial u)(y) = j^{\wedge}\Bl_G^D(u) = j^{\wedge}(0) = 0\]
This completes the proof that $\operatorname{Ext}_{\Z[G]}^1(H_1(M_K;\Z[G]),\Z[G])=~0$.
\end{proof}

Using the short exact sequence from Proposition~\ref{prop:necessary-condition}, we deduce the possible isomorphism classes for the Alexander module of a $G$-homotopy ribbon knot.

\begin{lemma}
\label{lem:ExtensionComputation}
Let $M$ be a $\LZ$-module, and let $P \subset M$ be a submodule that is isomorphic to one of $\Z[t^{\pm 1}]/(t-2)$ or~$\Z[t^{\pm 1}]/(2t-1)$ and fits into a short exact sequence
 \begin{equation}
0 \to P \to M \to \overline{P} \to 0.
\end{equation}
Then there are only two possible isomorphism classes of $\Z[t^{\pm 1}]$-modules for the central module in such an extension.  Indeed, $\Ext^1_{\Z[t^{\pm 1}]}(\overline{P},P) \cong \Z_3,$ the cyclic group of order $3$,
where
\begin{enumerate}
\item $0 \in \Z_3$ corresponds to the split extension with $M \cong \Z[t^{\pm 1}]/(t-2) \oplus \Z[t^{\pm 1}]/(2t-1)$,
\item $\pm 1 \in \Z_3$ correspond to the cyclic module $M \cong \Z[t^{\pm 1}]/(t-2)(2t-1)$.
\end{enumerate}
In particular, if $K$ is $G$-homotopy ribbon, then its Alexander module $H_1(M_K;\LZ)$ must belong to one of these isomorphisms types, and both cases are realised.
\end{lemma}
\begin{proof}
First, we compute the extension group $\Ext^1_{\Z[t^{\pm 1}]}(\overline{P},P)$ for $P = \Z[t^{\pm 1}]/(2t-1)$; the case~$P=\Z[t^{\pm 1}]/(t-2)$ is analogous.
We can use
\[0 \to \Z[t^{\pm 1}] \xrightarrow{t-2} \Z[t^{\pm 1}] \to \overline{P} \to 0\]
as a free $\Z[t^{\pm 1}]$-module resolution. Then we compute the abelian group:
\begin{align*}
& \Ext^1_{\Z[t^{\pm 1}]}(\overline{P},P) \cong \Ext^1_{\Z[t^{\pm 1}]}(\Z[t^{\pm 1}]/(t-2),\Z[t^{\pm 1}]/(2t-1)) \\
\cong &  \coker\big(\Hom_{\Z[t^{\pm 1}]}(\Z[t^{\pm 1}], \Z[t^{\pm 1}]/(2t-1)) \to \Hom_{\Z[t^{\pm 1}]}(\Z[t^{\pm 1}], \Z[t^{\pm 1}]/(2t-1)) \big)\\
\cong & \coker\big(\Z[t^{\pm 1}]/(2t-1) \xrightarrow{(t-2)} \Z[t^{\pm 1}]/(2t-1) \big) \cong \Z[t^{\pm 1}]/(t-2,2t-1).
\end{align*}
In this quotient, we have $t = t + (2-t)=2$, so $t^k = 2^k$.
Similarly, $t^{-1} = t^{-1} + t^{-1}(2t-1) =  2$ and so we also have $t^{-k} = 2^k$.
Therefore every element in this quotient can be expressed as a multiple of $1$.
We also note that $0 =  (2t-1) -  2(t-2) = 3$.
Moreover the resultant of $t-2$ and~$2t-1$ is~$\det \bsm 1 & 2 \\ -2 & -1 \esm =3$, so for the ideal $I :=(t-2,2t-1) \lhd \Z[t^{\pm 1}]$ we have~$I \cap \Z\langle 1 \rangle = (3) \lhd \Z$.
As a consequence, nothing more is killed in $\Ext^1_{\Z[t^{\pm 1}]}(\overline{P}_D,P_D) $ and we obtain the required result:
$$\Ext^1_{\Z[t^{\pm 1}]}(\overline{P},P) \cong \Z_3.$$
Next, we describe the extensions resulting from this computation.
The trivial element $0 \in~\Z_3$ corresponds as always to the split extension $\Z[t^{\pm 1}]/(t-2) \oplus \Z[t^{\pm 1}]/(2t-1)$.
On the other hand, the elements $\pm 1 \in \Z_3$  both correspond to the same module, namely the cyclic module~$\Z[t^{\pm 1}]/(t-2)(2t-1)$, but with different maps in the extension short exact sequence.
Here, note that the computation that $I \cap \Z = (3)$ above also implies that the split extension is not cyclic, as can be seen by comparing the second elementary ideals.

The assertion on $G$-homotopy ribbon knots now follows from the first item of Proposition~\ref{prop:necessary-condition}.
Finally, the knots~$K_0$ and~$K_{-1}$ from Example~\ref{thm:ExampleIntro} realise the two possibilities for the Alexander module. This completes the proof of the lemma.
\end{proof}

In the case of the two Alexander modules described in Lemma~\ref{lem:ExtensionComputation}, the next lemma shows that at most two submodules can arise as lagrangians induced by $G$-homotopy ribbon discs.

\begin{lemma}
\label{lem:TwoSubmodules}
The following two assertions hold.
\begin{enumerate}
\item If $P \subset M:=\LZ/(t-2)(2t-1)$ is a submodule that is abstractly $P \cong \LZ/(2t-1)$ $($resp.\ $P \cong \LZ/(t-2))$ and fits into a short exact sequence
$$ 0 \to P \to M \to \overline{P} \to 0,$$
then $P=(2t-1)M$ $($resp.\  $P=(t-2)M)$.
\item If $P \subset \LZ/(2t-1) \oplus \LZ/(t-2)$ is a submodule that is abstractly $P \cong \LZ/(2t-~1)$ $($resp.\ $P \cong \LZ/(t-2))$ and fits into a short exact sequence
$$ 0 \to P \to M \to \overline{P} \to 0,$$
then $P=\LZ/(2t-1) \oplus \lbrace 0\rbrace $ $($resp.\  $P=\lbrace 0 \rbrace \oplus \LZ/(t-2))$.
\end{enumerate}
\end{lemma}
\begin{proof}
We prove the first assertion for $P \cong \LZ/(2t-1)$; the proof of the second case is identical.
Using the definition of $M:=\LZ/(t-2)(2t-1)$, we see that $P \subset (t-2)M$.
As~$M/P \cong \LZ/(t-2)$, we have $[(t-2)] = 0$ in $M/P$ and therefore $[t-2] \in P \subset M$, so that~$P \supset (t-2)M$, concluding the proof of the first assertion.

We prove the second assertion for $P \cong \LZ/(2t-1)$; the proof of the second case is identical.
We claim that $P \subset \LZ/(2t-1) \oplus \lbrace 0\rbrace$.
Since $P \cong \LZ/(2t-1)$, for $p=([p_1],[p_2]) \in P$, we have $(2t-1)([p_1],[p_2])=0$ and in particular $[(2t-1)p_2]=0$ in~$\LZ/(t-2)$.
This implies that $(2t-1)p_2 =(t-2)x$ for some $x \in \LZ$.
Since $\LZ$ is a unique factorization domain and since $(2t-1)$ and~$(t-2)$ are coprime polynomials, we deduce that~$p_2=(t-2)z$ for some $z \in \LZ$.
It follows that $[p_2]=0$ in $\LZ/(t-2)$ and therefore~$p \in \LZ/(2t-1) \oplus \lbrace 0\rbrace$, concluding the proof of the claim.

Since we also assumed that $M/P \cong \LZ/(t-2)$, the claim implies that
$$ \LZ/(t-2)=M/P=(\LZ/(2t-1))/P \oplus \LZ/(t-2).$$
Tensoring with $\Q$ and using that $\Q[t^{\pm 1}]$-modules admits primary decompositions, we deduce that $(\LZ/(2t-1))/P \otimes_\Z \Q=0.$
This implies that $(\LZ/(2t-1))/P$ is $\Z$-torsion and therefore that $M/P$ contains $\Z$-torsion.
But $M/P\cong \Z[\frac{1}{2}]$ is $\Z$-torsion free, so we deduce that $(\LZ/(2t-1))/P=0$ and consequently that $P=\LZ/(2t-1) \oplus \lbrace 0 \rbrace$ as desired.
\end{proof}

Let $K$ be an oriented knot, and let $P \subseteq H_1(M_K;\Z[t^{\pm 1}])$ be a submodule of
$H_1(M_K;\Z[t^{\pm 1}])$ which is isomorphic to either one of the two submodules $\Z[t^{\pm 1}]/(t-2)$ or $\Z[t^{\pm 1}]/(2t-1)$ and such that $H_1(M_K;\Z[t^{\pm 1}])/P \cong \overline{P}$.
In particular, $\overline{P}$ is again isomorphic to $\Z[\tmfrac{1}{2}]$, for one of the module structures.
As mentioned in the introduction, there is an associated homomorphism
\[\phi_P \colon \pi_1(M_K) \twoheadrightarrow \pi_1(M_K)/\pi_1(M_K)^{(2)} \cong \Z \ltimes H_1(M_K;\Z[t^{\pm 1}]) \twoheadrightarrow \Z \ltimes H_1(M_K;\Z[t^{\pm 1}])/P \cong G.\]
Note that if $P=P_D$ for some homotopy ribbon disc $D$, then $\phi_P$ coincides with the homomorphism induced by the inclusion $M_K \hookrightarrow N_D$.

\begin{theorem}\label{thm:characterisation}
Let $K$ be an oriented knot, and let~$\mathcal{L}$ be the set of submodules $P \subseteq H_1(M_K;\LZ)$ of the Alexander module that are isomorphic to one of~$\Z[t^{\pm 1}]/(t-2)$ or~$\Z[t^{\pm 1}]/(2t-1)$ and fit into a short exact sequence
 \begin{equation}
 \label{eq:AwesomeAssumption0}
0 \to P \to H_1(M_K;\Z[t^{\pm 1}]) \to \overline{P} \to 0.
\end{equation}
Mapping a $G$-homotopy ribbon disc to its induced lagrangian gives rise to a bijection between
\begin{itemize}
\item $G$-homotopy ribbon discs for $K$, up to topological ambient isotopy rel.\ boundary;
\item submodules $P \in \mathcal{L}$ such that, with respect to $\phi_P$,
\begin{equation*}
\tag{Ext} \label{eq:ExtCondition}
\operatorname{Ext}_{\Z[G]}^1(H_1(M_K;\Z[G]),\Z[G])=0.
\end{equation*}
\end{itemize}
Moreover, these sets have cardinality at most two.
\end{theorem}
\begin{proof}
First we show that assigning to a slice disc its induced lagrangian determines a map from the first set to the second set in the statement of the theorem.
 Let $D$ be a $G$-homotopy ribbon disc for $K$. Let  $P=P_D$ be the induced lagrangian.
   In this case, up to an isomorphism of $G$, the map $\phi_P$ coincides with the inclusion induced map~$\pi_1(M_K) \twoheadrightarrow \pi_1(N_D) = G$.
   As a consequence, the first item of Proposition~\ref{prop:necessary-condition} ensures that the lagrangian $P=P_D$ belongs to~$\mathcal{L}$,
while the second item of Proposition~\ref{prop:necessary-condition} guarantees that~$\operatorname{Ext}_{\Z[G]}^1(H_1(M_K;\Z[G]),\Z[G])=~0$.  Therefore the assignment determines a map from the first to the second step as asserted.

Next, by Theorem~\ref{thm:Intro}, $D$ is determined up to topological ambient isotopy rel.\ boundary by the induced lagrangian $P=P_D$. It follows that the assignment is injective.

Now we prove surjectivity.
Given a submodule $P \in \mathcal{L}$, we obtain the surjective homomorphism $\phi_P \colon \pi_1(M_K) \twoheadrightarrow G$.
Since, with respect to $\phi_P$, we assumed that the Ext condition $\operatorname{Ext}_{\Z[G]}^1(H_1(M_K;\Z[G]),\Z[G])=0$ holds, the second part of Theorem~\ref{thm:existence-conditions} (which is \cite[Theorem~1.3]{FriedlTeichner}) ensures the existence of a $G$-homotopy ribbon disc $D$ for~$K$ with $P=P_D$.
This establishes that the assignment is a bijection.
Finally, Lemma~\ref{lem:TwoSubmodules} shows that if $\mathcal{L}$ is nonempty then it contains precisely two elements $P,\overline{P}$.
It follows that $K$ has at most two $G$-homotopy ribbon discs up to topological ambient isotopy rel.\ boundary.
This completes the proof of Theorem~\ref{thm:characterisation}.
\end{proof}

\section{Examples}
\label{sec:Examples}

Throughout this section, we set~$G:=\Z \ltimes \Z[\tmfrac{1}{2}]$.
Given $n \in \Z$, consider the knot~$K_{n}$ obtained by adding~$n$ full twists in the left band of the~$9_{46}$ knot as on the left hand side of Figure~\ref{fig:Kn} below.
The goal of this section is to use Theorem~\ref{thm:Intro} to study the $G$-homotopy ribbon discs of $K_{n}$.
\medbreak
\begin{figure}[!htb]
\centering
\captionsetup[subfigure]{}
\begin{subfigure}[t]{.30\textwidth}
\labellist
    \small
    \pinlabel {$n$} at 105 480
    \endlabellist
\includegraphics[scale=0.2]{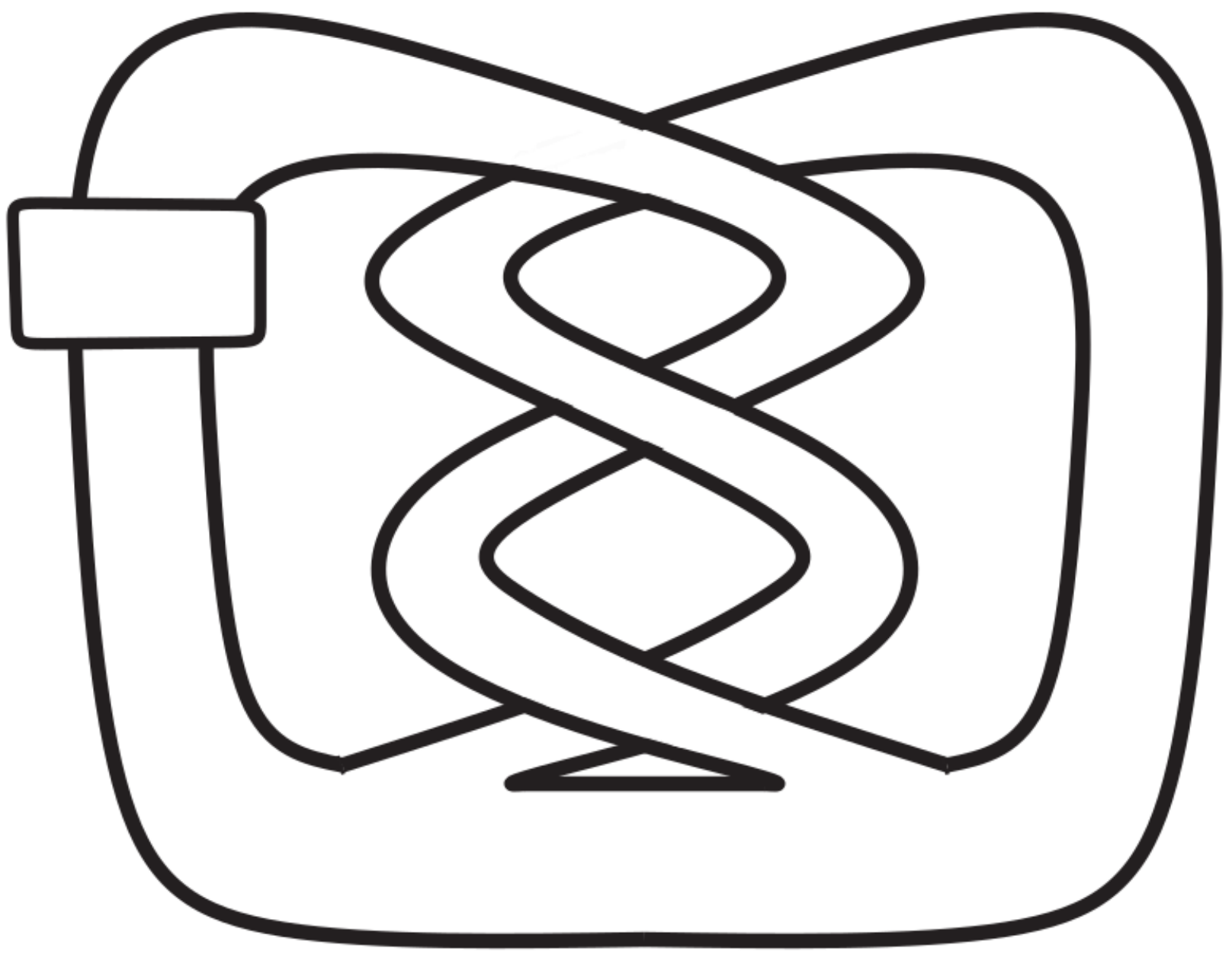}
\end{subfigure}
 \hspace{2.5cm}
\begin{subfigure}[t]{.30\textwidth}
\labellist
    \small
    \pinlabel {$n$} at 103 476
    \pinlabel {$\alpha$} at 0 260
    \pinlabel {$\beta$} at 780 260
    \pinlabel {$a$} at 140 120
    \pinlabel {$b$} at 660 120
    \endlabellist
\includegraphics[scale=0.2]{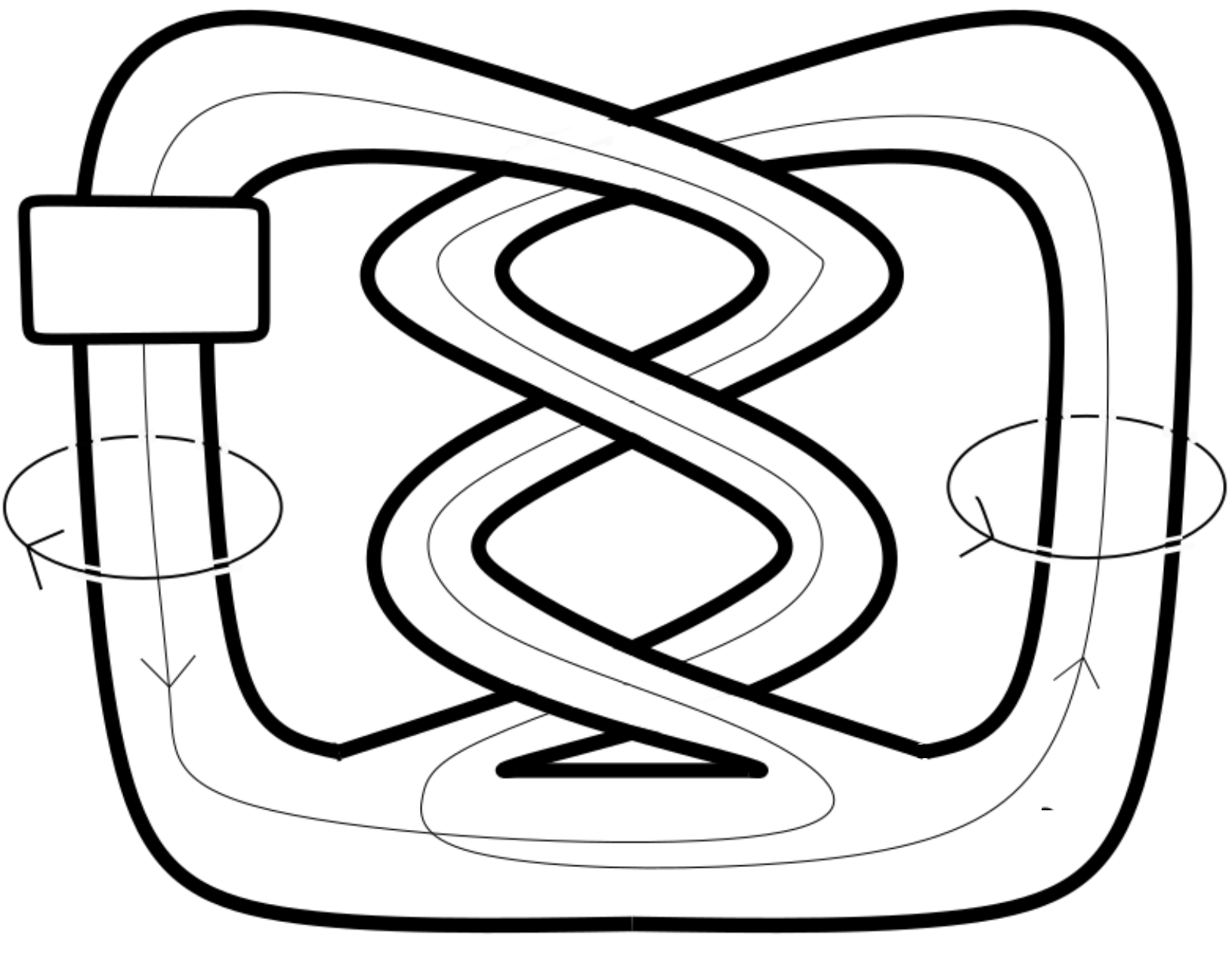}
\end{subfigure}
\caption{On the left: the knot~$K_{n}$; on the right: a Seifert surface $F_{n}$ for~$K_{n}$ as well as (oriented curves representing) generators $a,b$ of $H_1(F_{n};\Z)$ and their Alexander dual curves $\alpha,\beta$.}
\label{fig:Kn}
\end{figure}
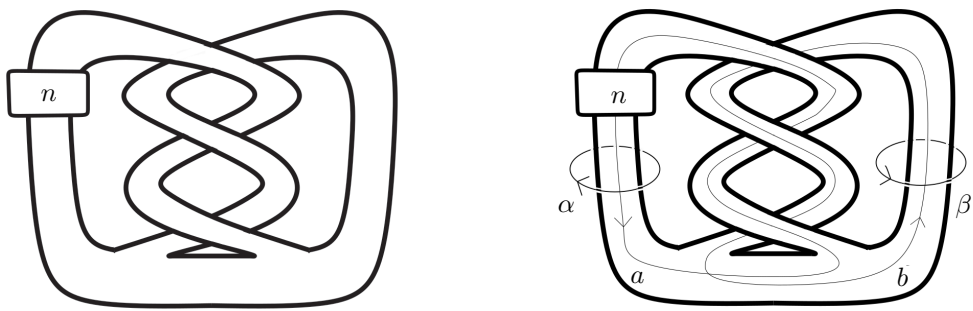

Let~$F:=F_{n}$ be the obvious Seifert surface for~$K:=K_{n}$ depicted on the right hand side of Figure~\ref{fig:Kn}.
This figure also shows simple closed curves~$\alpha,\beta \subset S^3 \setminus F$  Alexander dual to generators $a,b$ of~$H_1(F;\Z)$, which are also shown.
These loops $\alpha$ and $\beta$ (or more precisely their lifts to the infinite cyclic cover of~$M_K$) generate~$H_1(M_K;\Z[t^{\pm 1}])$ as a~$\Z[t^{\pm 1}]$-module.

\subsection{The case that $n$ is a multiple of $3$}\label{subsection:n-mult-of-3}

Now we restrict to the case that $n=3k$ for some $k \in \Z$.  In this case we are able to classify the $G$-homotopy ribbon discs for $K_{3k}$.

We write homology classes without brackets and we set~$\beta':=k\alpha+\beta$ so that a Seifert matrix computation yields
\[
H_1(X_K;\Z[t^{\pm 1}])=
\Z[t^{\pm 1}] \alpha /(t-2)\alpha \oplus \Z[t^{\pm 1}]\beta'/(2t-1)\beta'.\]
A \emph{metabolizer}~$\mathfrak{m}$ for~$K$ is a rank~$1$ summand of~$H_1(F;\Z)\cong \Z^2$ on which the Seifert form vanishes.
Following \cite[Definition~5.4]{CochranHarveyLeidyDerivative}, a
metabolizer~$\mathfrak{m}$ \emph{represents} a lagrangian~$P$ for the rational Blanchfield pairing if the image of~$\mathfrak{m}$ under the map
$$ H_1(F;\Z) \to H_1(F;\Z) \otimes \Q \stackrel{i_*}{\twoheadrightarrow} H_1(X_K;\Q[t^{\pm 1}])~$$
spans~$P$ as a~$\Q$-vector space; here $i_*$ is obtained by fixing a lift of $F$ to the infinite cyclic cover of $X_K$.
The next lemma describes the lagrangians of~$\Bl(K)$ as well as their generators and metabolizers which represent them.

\begin{lemma}
\label{lem:PropertiesKn}
The Blanchfield pairing~$\Bl(K_{3k})$ admits precisely two distinct lagrangians~$P_1,P_2$ that are respectively generated by~$\alpha$ and~$\beta'=k \alpha+\beta$.
The lagrangian~$P_2$ is represented by the metabolizer $\Z \langle a-kb\rangle \subset H_1(F;\Z)$.
\end{lemma}

\begin{proof}
The description of the lagrangians for $\Bl(K_{3k})$ and their generators can be found in~\cite[p.~4--5]{FriedlTeichnerErratum} (the unpublished clarification of the published erratum to \cite{FriedlTeichner}).
To prove the last statement, we use Cochran, Harvey and Leidy's constructive proof of the fact that every lagrangian is represented by a metabolizer~\cite[Lemma~5.5]{CochranHarveyLeidyDerivative}.
We start from the lagrangian~$P_2=\langle k \alpha +\beta \rangle$, viewed as a $1$-dimensional~$\Q$-vector subspace of the rational Alexander module~$\mathcal{A}_0(K):=H_1(X_K;\Q[t^{\pm 1}]) \cong \Q^2$.
In the notation of~\cite{CochranHarveyLeidyDerivative}, the element~$a_1:=a-kb$ maps to $\gamma_1:=k\alpha +\beta$ under the inclusion induced map
\[H_1(F;\Z) = H_1(F \times \{1\};\Z) \to H_1(S^3 \sm (F \times (-1,1));\Z),\]
which with respect to the bases $\{a,b\}$ and $\{\alpha,\beta\}$ respectively is represented by the Seifert form $\begin{pmatrix} 3k & 2 \\ 1 & 0 \end{pmatrix}$.
Cochran, Harvey and Leidy then prove that~$\lbrace a_1 \rbrace$ spans~$P_2$ in the rational vector space $\mathcal{A}_0(K)$~\cite[p.760-761]{CochranHarveyLeidyDerivative}.
This concludes the proof of the lemma.
\end{proof}

Although we do not require this fact, observe that the same argument as in the proof of Lemma~\ref{lem:PropertiesKn} shows that the lagrangian $P=P_1=\langle \alpha \rangle$ is represented by the metabolizer $\Z \langle b \rangle$.

The next result provides an application of Theorem~\ref{thm:Intro}.

\begin{theorem}
\label{thm:Kn}
Set $G:=\Z \ltimes \Z[\tmfrac{1}{2}]$.
Up to ambient isotopy rel.\ boundary, the knot~$K_{3k}$ admits
\begin{enumerate}
\item precisely two distinct~$G$-homotopy ribbon discs if~$k=0,-1$;
\item a unique~$G$-homotopy ribbon disc if~$k \neq 0,-1.$
\end{enumerate}
\end{theorem}

\begin{proof}
\begin{figure}[!htb]
\centering
\labellist
    \small
    \pinlabel {$3k$} at 107 477
    \pinlabel {$0$} at 35 310
    \endlabellist
\includegraphics[scale=0.2]{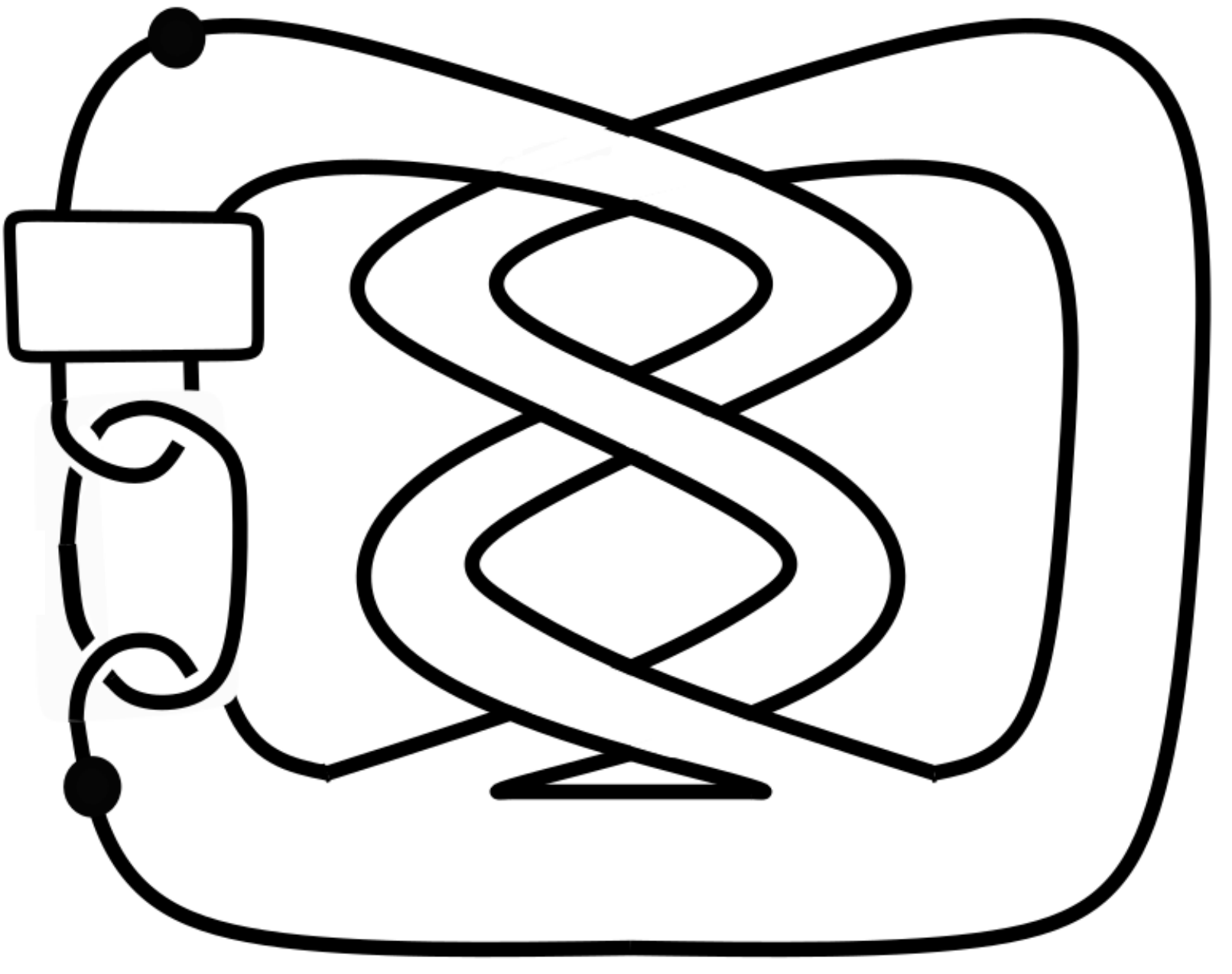}
\caption{A Kirby diagram for the exterior $N_{D_1}$ of the homotopy ribbon disc~$D_1$ obtained by performing a saddle move on the left band of $K_{3k}$.}
\label{fig:Kirby}
\end{figure}
Throughout the proof, we write~$K:=K_{3k}$.
We first assume that~$k=0$. We will give full details for $k=0$, and adapt them to the case~$k=-1$ below.
Performing a saddle move on the left (resp.\ right) band of~$K$ gives rise to a ribbon disc~$D_1$ (resp.~$D_2$).
\begin{claim}
The discs~$D_1$ and~$D_2$ are~$G$-homotopy ribbon and respectively induce the lagrangians~$P_1$ and~$P_2$ described in Lemma~\ref{lem:PropertiesKn}.
\end{claim}
\begin{proof}
We only prove this claim for~$D_1$, since~$D_2$ can be treated similarly.
We draw a Kirby diagram of~$N_{D_1}$ as in Figure~\ref{fig:Kirby}; we refer to~\cite[p.~213]{GompfStipsicz} for details on this procedure.
The group~$\pi_1(N_{D_1})$ admits a presentation with two generators, the meridians~$a,b$ of the dotted circles, and a unique relation~$bab^{-1}a^{-1}b^{-1}a^{-1}$, obtained by reading off the word described by the 2-handle.
Setting~$c:=ab$, we deduce that~$D_1$ is~$G$-ribbon:
$$ \pi_1(N_{D_1}) \cong  \langle a,b \mid \ bab^{-1}a^{-1}b^{-1}a^{-1}=1 \rangle \cong \langle a,c \ | \ a^{-1}ca=c^2 \rangle  \cong G.$$
Since ribbon discs are homotopy ribbon, we have proved that~$D_1$ is~$G$-homotopy ribbon.
Next, we show that $D_1$ induces $P_1=\langle \alpha \rangle$.
As explained at the beginning of this section, the Alexander module~$H_1(X_K;\Z[t^{\pm 1}])$ is generated by (homology classes of) the curves~$\alpha$ and~$\beta$ depicted in the right hand side of Figure~\ref{fig:Kn}.
After straightening the dotted circles in the Kirby diagram of~$N_{D_1}$, one sees that~$(\iota_{D_1})_*$ maps~$\alpha$ to zero and maps~$\beta$ to~$c$.
Since Lemma~\ref{lem:PropertiesKn} implies that~$\Bl(K)$ admits precisely two lagrangians,~$P_{D_1}$ must equal either~$P_1=\langle \alpha \rangle$ or~$P_2=\langle \beta \rangle$.
Since we established that~$\alpha$ lies in~$P_{D_1}$ but~$\beta$ does not, we deduce that~$P_{D_1}=P_1$.
This concludes the proof of the claim.
\end{proof}

Using the claim, in order to establish the result in the~$k=0$ case, it remains to show that~$D_1$ and~$D_2$ are distinct and that, up to ambient isotopy, there are no other~$G$-homotopy ribbon discs.
First, assume that~$D$ induces~$P_1$ and~$D'$ induces~$P_2$; we claim that~$D$ and~$D'$ are not ambiently isotopic rel.\ boundary.
 By means of contradiction, assume they are.
Using Lemma~\ref{lem:Prelim}, this ambient isotopy induces a rel.\ boundary homeomorphism of~$D^4$.
In particular this homeomorphism is the identity on~$X_K$.
Lifting these considerations to the infinite cyclic covers, it follows that~$P_1=P_2$. This is a contradiction and proves the claim that~$D$ and~$D'$ are not ambiently isotopic rel.\ boundary. Finally, we show that there are no other~$G$-homotopy ribbon discs than~$D_1$ and~$D_2$. If~$D$ is such disc, then Lemma~\ref{lem:PropertiesKn} implies that it must induce either~$P_1$ or~$P_2$. Without loss of generality, assume that~$D$ induces~$P_1$. By Theorem~\ref{thm:Intro}, since~$D_1$ and~$D$ induce the same lagrangian, they must be ambiently isotopic rel.\ boundary.

When $k=-1$, the lagrangian $P_2$ is represented by the metabolizer $\Z \langle a+b \rangle $, and $a+b$ is represented by the unknotted curve $J$ depicted on the left hand side of Figure~\ref{fig:K-3}.
The argument works similarly to the case $k=0$, after performing an isotopy on $F$ (resulting in the surface $F'$ depicted on the right hand side of Figure~\ref{fig:K-3}) so that~$J$ becomes the core of one the two bands of $F'$.
\begin{figure}[!htb]
\centering
\captionsetup[subfigure]{}
\begin{subfigure}[t]{.30\textwidth}
\labellist
\small
 \pinlabel {$-3$} at 69 509
    \endlabellist
    \includegraphics[scale=0.3]{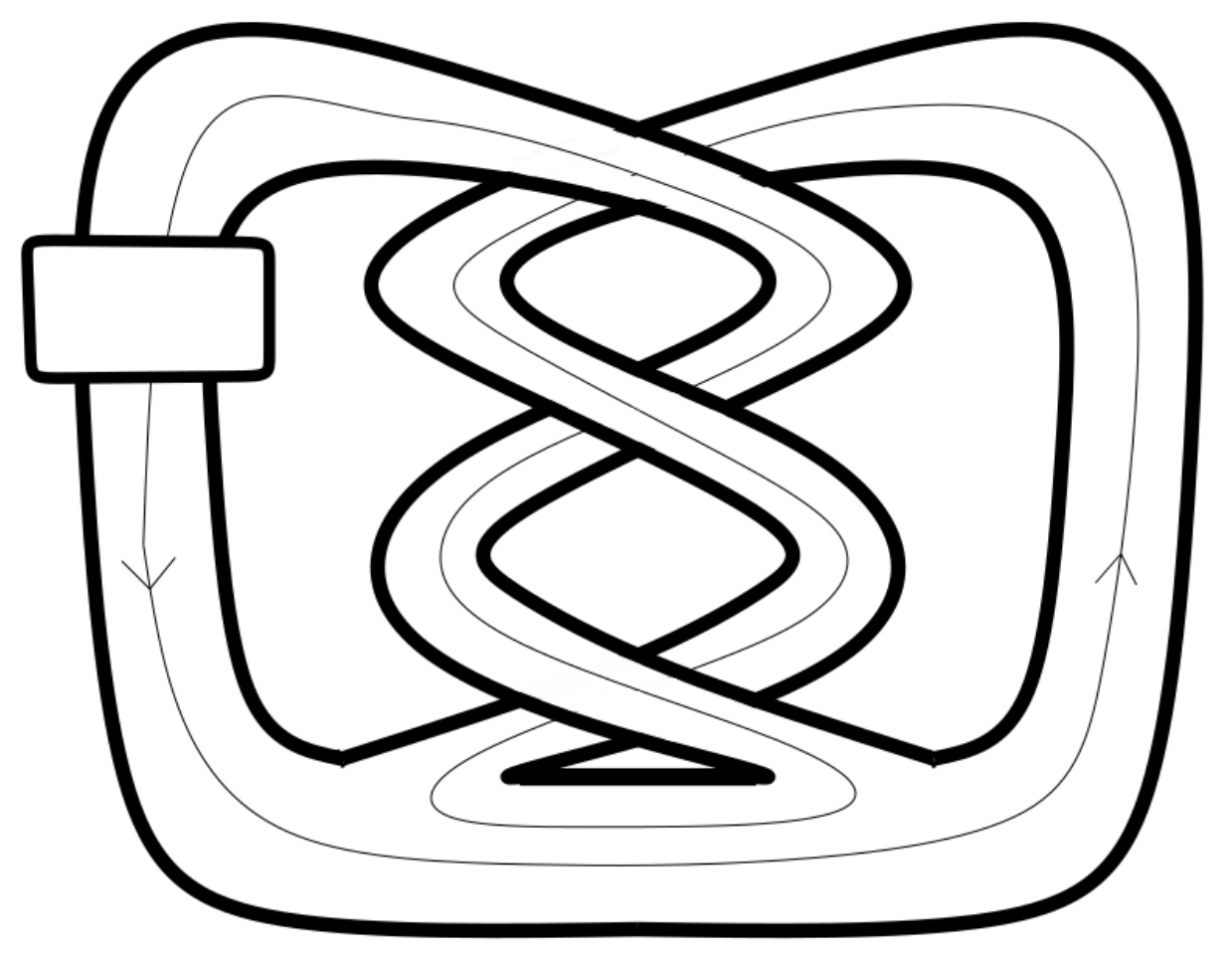}
\end{subfigure}
 \hspace{2cm}
\begin{subfigure}[t]{.30\textwidth}
\labellist
\small
 \pinlabel {$-3$} at 48 550
    \endlabellist
\includegraphics[scale=0.29]{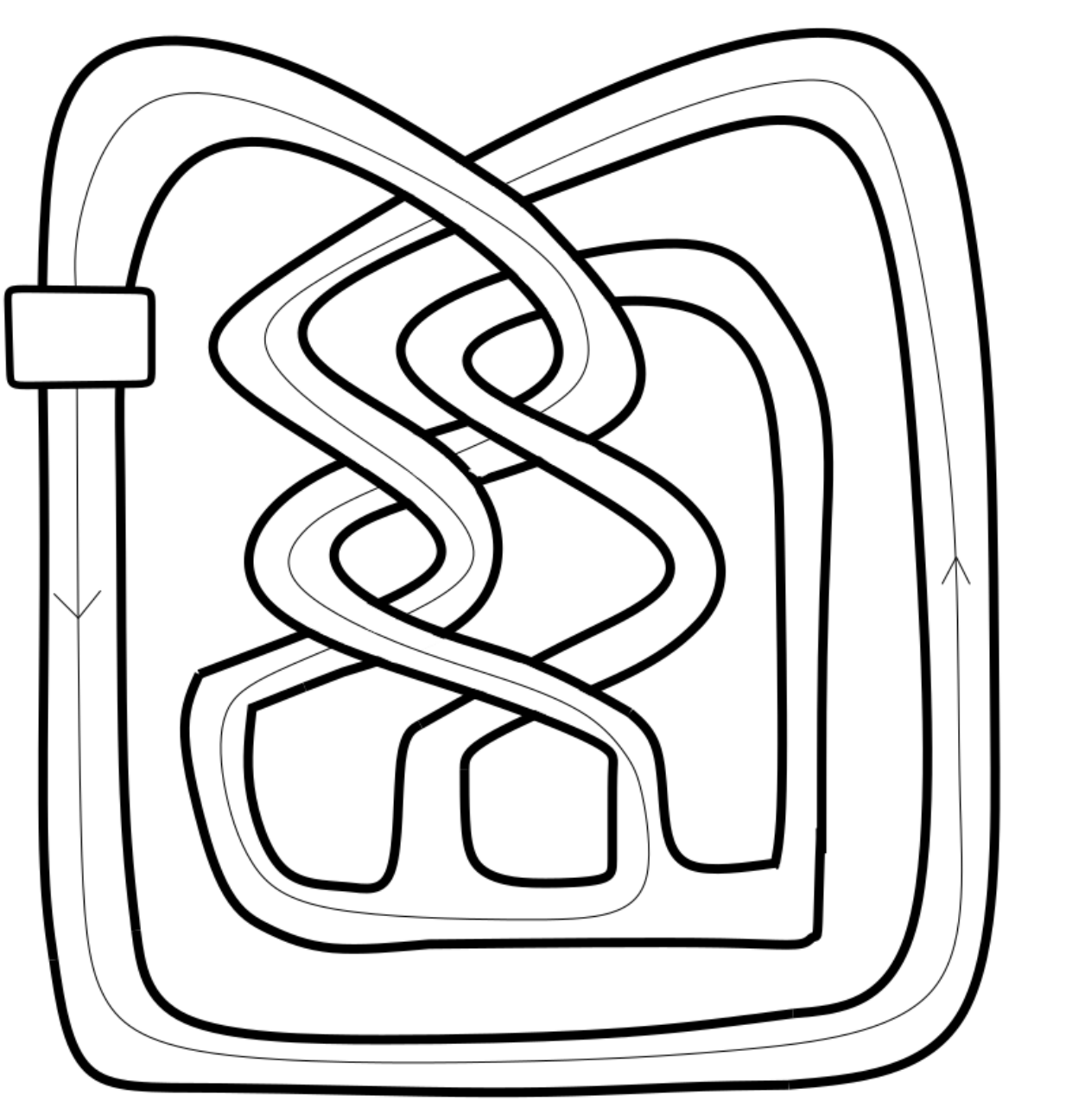}
\end{subfigure}
\caption{On the left: the surface $F$ together with the curve $J$ which represents the homology class $a+b$.
On the right: the surface $F'$ obtained from $F$ by performing an isotopy so that~$J$ becomes the core of one of the bands.}
\label{fig:K-3}
\end{figure}

Finally, we assume that~$k \neq 0,-1$.
Arguing as in the~$k=0$ case and applying Theorem~\ref{thm:Intro}, we know that up to ambient isotopy rel.\ boundary,~$K$ admits at most two~$G$-homotopy ribbon discs, corresponding to the lagrangians~$P_1$ and~$P_2$ described in Lemma~\ref{lem:PropertiesKn}.
As in the previous paragraphs, a saddle move on the left band of~$K$ produces a~$G$-homotopy ribbon disc that induces~$P_1$.
\begin{claim}
The lagrangian~$P_2=\langle k \alpha + \beta \rangle$ is not induced by any slice disc.
\end{claim}
\begin{proof}
Recall that a metabolizer~$\mathfrak{m}$ of the Seifert form represents a lagrangian~$P$ for the rational Blanchfield pairing if the image of~$\mathfrak{m}$ under the map
$$ H_1(F;\Z) \to H_1(F;\Z) \otimes \Q \twoheadrightarrow H_1(M_K;\Q[t^{\pm 1}])~$$
spans~$P$ as a~$\Q$-vector space.
Following \cite[Definition 5.1]{CochranHarveyLeidyDerivative} a \emph{derivative of~$K$ with respect to~$\mathfrak{m}$} is a knot~$J$ embedded in~$F$ that gives a basis for~$\mathfrak{m}$.
Lemma~\ref{lem:PropertiesKn} establishes that $P_2$ is represented by the metabolizer~$\mathfrak{m}:=\Z \langle a-kb \rangle \subset H_1(F;\Z)$.
Reading braids from bottom to top, for $k>0$, a derivative of~$K$ with respect to $\mathfrak{m}$ is given by the negative braid knot~$J_k=\widehat{\gamma}_k$, where $\gamma_k$ is the negative braid
$$\gamma_k=(\sigma_k^{-1}\cdots \sigma_1^{-1})(\sigma_1^{-1}\cdots \sigma_k^{-1})(\sigma_k^{-1}\cdots \sigma_1^{-1}).$$
For $k=2$, this knot is depicted in Figure~\ref{fig:Derivative}; note also that for $k=0,-1$, the derivative is unknotted, as expected.
For $k<-1$, the derivative is instead given by~$J_{-k-1}$.

\begin{figure}[!htb]
\centering
\labellist
        \small
    \pinlabel {$6$} at 75 510
\endlabellist
\includegraphics[scale=0.25]{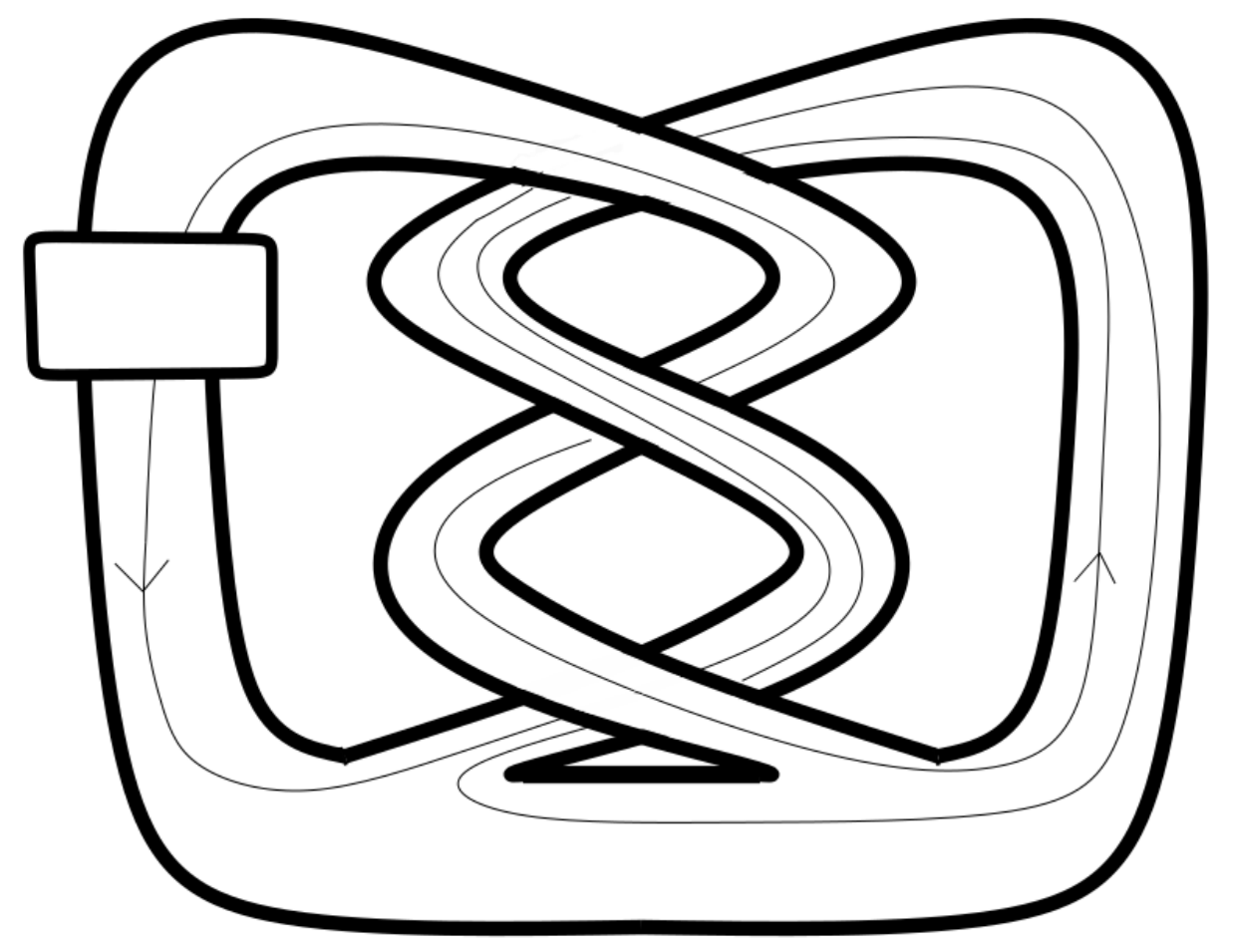}
\caption{The knot $J_2$ on the surface $F$.}

\label{fig:Derivative}
\end{figure}
Next, we consider the first order signature~$\rho^1(K,\phi_{P_2})$ associated to the lagrangian~$P_2$ of~$\Bl(K)$.
Since we need only two properties of~$\rho^1(K,\phi_{P_2})$, we omit its definition but refer the interested reader to ~\cite[Definition~4.1]{CochranHarveyLeidyDerivative} for details.
Use~$\rho^0(J_k)$ to denote the integral of the Levine-Tristram signature function~$\sigma_{J_k}(\omega)$ over~$S^1$.
Since~$J_k$ is a negative braid knot, we have~$\sigma_{J_k}(\omega) \geq 0$ for all~$\omega \in S^1$ (e.g.\ negative braid knots can be unknotted using only negative to positive crossing changes) and~$\sigma_{J_k}(-1)>0$ (see e.g.~\cite{Rudolph} or~\cite{Przytycki}).
Combining this observation with~\cite[Corollary~5.8]{CochranHarveyLeidyDerivative} implies that
\[\rho^1(K,\phi_{P_2})=\rho^0(J_k)>0.\]
To finish the proof, if~$P_2$ were induced by a slice disc~$D$, then~\cite[Theorem 4.2]{CochranHarveyLeidyDerivative} would imply that~$\rho^1(K,\phi_{P_2})=0$, a contradiction. This concludes the proof of the claim that the lagrangian~$P_2=\langle k \alpha + \beta \rangle$ is not induced by a slice disc.
\end{proof}

Summarising, when $k \neq 0,-1$, we know that~$P_1$ is induced by a slice disc $D$, but that~$P_2$ is not induced by any slice disc.
The fact that $D$ is unique up to ambient isotopy rel.\ boundary now follows by applying Theorem~\ref{thm:Intro}.
This concludes the proof of Theorem~\ref{thm:Kn}.
\end{proof}

\subsection{The cases with $n$ not a multiple of 3}\label{subsection:n-not-mult-of-3}

Now  we study the cases that $n$ is not a multiple of 3. Define $k \in \Z$ and $x \in \{1,2\}$ as the unique numbers with $n = 3k +x$.

As above, let~$F:=F_{n}$ be the obvious Seifert surface for~$K:=K_{n}$ depicted on the right hand side of Figure~\ref{fig:Kn}.
This figure also shows simple closed curves~$\alpha,\beta \subset S^3 \setminus F$ Alexander dual to generators $a,b$ of~$H_1(F;\Z)$. The loops $\alpha$ and $\beta$ generate~$H_1(M_K;\Z[t^{\pm 1}])$.
A computation with the Seifert matrix shows that
\[H_1(M_K;\Z[t^{\pm 1}]) \cong \frac{\Z[t^{\pm 1}]}{(t-2)(2t-1)}\]
is a cyclic $\Z[t^{\pm 1}]$-module generated by $k\alpha+\beta$.
Using \cite[Theorem~1.4]{FriedlPowell-Bl-computation}, we compute that the Blanchfield form is isometric to:
\begin{align*}
  \frac{\Z[t^{\pm 1}]}{(t-2)(2t-1)} \times \frac{\Z[t^{\pm 1}]}{(t-2)(2t-1)} &\to \Q(t)/\Z[t^{\pm 1}] \\
  (p,q) & \mapsto  \frac{- p x(t-1)^2 \overline{q}}{(t-2)(2t-1)}.
\end{align*}
  Contrary to the statement in \cite[p.~4--5]{FriedlTeichnerErratum} (the unpublished clarification of the published erratum to \cite{FriedlTeichner}), there are \emph{two} lagrangians for the Blanchfield form, namely the submodules
\[P_1 := (t-2)\Z[t^{\pm 1}] \text{ and } P_2:= (2t-1)\Z[t^{\pm 1}].\]
Here $P_1$ is generated by $\alpha$ and $P_2$ is generated by $n\alpha + 3\beta$.
To see that these are distinct submodules, note that if they were equal then there would exist $p,q \in \Z[t^{\pm 1}]$ such that $2t-1 = p(t-2) + q(t-2)(2t-1) = (t-2)(p+q(2t-1))$. But then multiplication of Laurent polynomials leads to addition of their widths, so $p+q(2t-1)$ is a monomial $\pm t^m$. But there is no monomial such that $2t-1 = \pm t^m(t-2)$.   It follows that $P_1$ and $P_2$ are indeed distinct lagrangian submodules.

Corresponding to these lagrangians of $\Bl(K)$ are derivative curves on $F$ representing~$b$ and~$3a-nb$ respectively. One can find these metabolizers directly by computing with the Seifert matrix~$\bp n & 2 \\ 1 & 0 \ep$.
For every $n$, as in Section~\ref{subsection:n-mult-of-3}, $b$ is represented by an unknotted, and therefore slice derivative curve, so there is an essentially unique slice disc corresponding to~$P_1$ for every~$n$.

The following proposition classifies the $G$-homotopy ribbon discs for small values of $n$.

\begin{proposition}\label{prop:other-n}
Set $G:=\Z \ltimes \Z[\tmfrac{1}{2}]$. Up to ambient isotopy rel.\ boundary,
  \begin{enumerate}
    \item the knots $K_{-1}$ and $K_{-2}$ admit precisely two distinct~$G$-homotopy ribbon discs;
     \item\label{item:part-2-prop-other-n} the knots $K_{-5}$, $K_{-4}$, $K_1$, and $K_2$ admit a unique~$G$-homotopy ribbon disc.
  \end{enumerate}
\end{proposition}

\begin{proof}
As described above, there is a slice disc corresponding to $P_1$.  For $n=-1,-2$, the other derivative curve, representing $3a+b$ and $3a+2b$ respectively, is also unknotted. In these cases there is therefore also a slice disc corresponding to the lagrangian $P_2$, and so by Theorem~\ref{thm:Intro} we have precisely two distinct~$G$-homotopy ribbon discs as claimed.

For $n \in \{-5,-4,1,2\}$, we drew the derivative curves $J_n$ on $F$ for~$3a-nb \in H_1(F;\Z)$, and used a computer\footnote{We used SnapPy to obtain the PD code of the $J_n$, Sage to deduce Seifert matrices, and Mathematica to deduce that the integral of the Levine-Tristram signature is
negative for $J_1,J_2$ and
positive for $J_{-4},J_{-5}$.} to show that $\rho^0(J_n)$,
 the integral over $S^1$ of the Levine-Tristram signature function~$\sigma_{J_n}(\omega)$, is nonzero.
 Thus by~\cite[Theorem 4.2]{CochranHarveyLeidyDerivative}, as explained in the proof of Theorem~\ref{thm:Kn},
there can be no slice disc corresponding to the lagrangian $P_2$.
It follows from Theorem~\ref{thm:Intro} that there is a  unique~$G$-homotopy ribbon disc for $K_n$ with $n \in \{-5,-4,1,2\}$.
\end{proof}

As mentioned in the introduction, we conjecture that for each $n$ with $n>0$ or $n< -3$, there is a unique~$G$-homotopy ribbon disc for $K_n$.  We have been unable to establish the required lower bounds on the absolute value of the integral of the signatures for the derivative curves corresponding to the lagrangian~$P_2$. It is encouraging that for the examples we checked with a computer, our conjecture holds. For larger absolute values of $n$, the derivatives become more complicated, so it seems doubtful that their signatures become trivial.

\section{Relaxing the rel.\ boundary restriction}\label{sec:relaxing}

In this section, we consider relaxing the rel.\ boundary condition.
Note that the two $G$-homotopy ribbon discs for $9_{46}$ are isotopic as disc knots. That is, if isotopies of the knot in~$S^3$ are also permitted, then $R:= 9_{46}$ admits an essentially unique $G$-homotopy ribbon disc.
\begin{figure}[!htb]
\centering
\captionsetup[subfigure]{}
\begin{subfigure}[t]{0.3\textwidth}
\centering
\labellist
    \small
    \pinlabel {$\eta_1$} at -20 455
        \pinlabel {$\eta_2$} at 615 455
    \endlabellist
\includegraphics[scale=0.25]{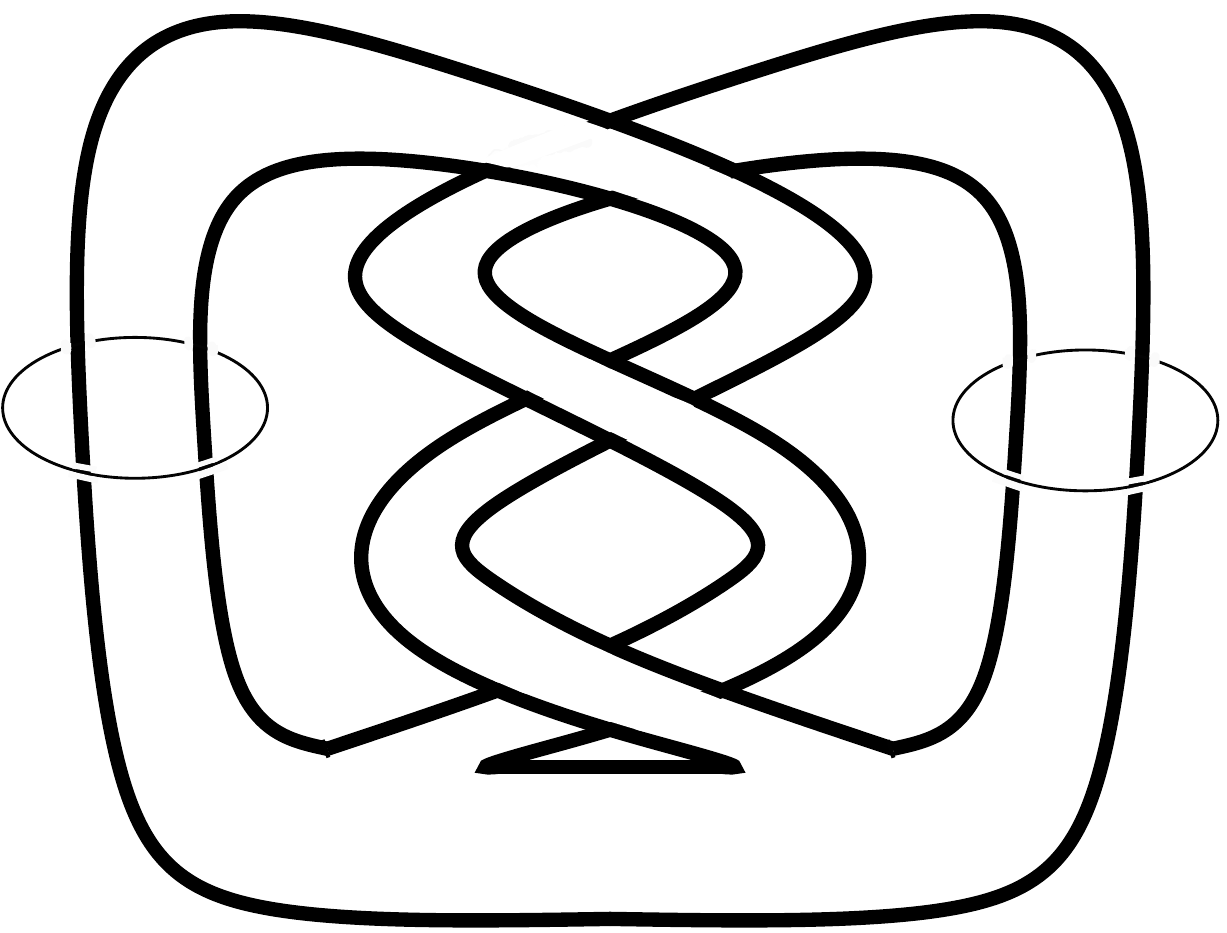}
\end{subfigure}
 \hspace{2.5cm}
\begin{subfigure}[t]{.30\textwidth}
\labellist
    \small
    \pinlabel {$J_1$} at 63 459
    \pinlabel {$J_2$} at 528 459
    \endlabellist
\includegraphics[scale=0.25]{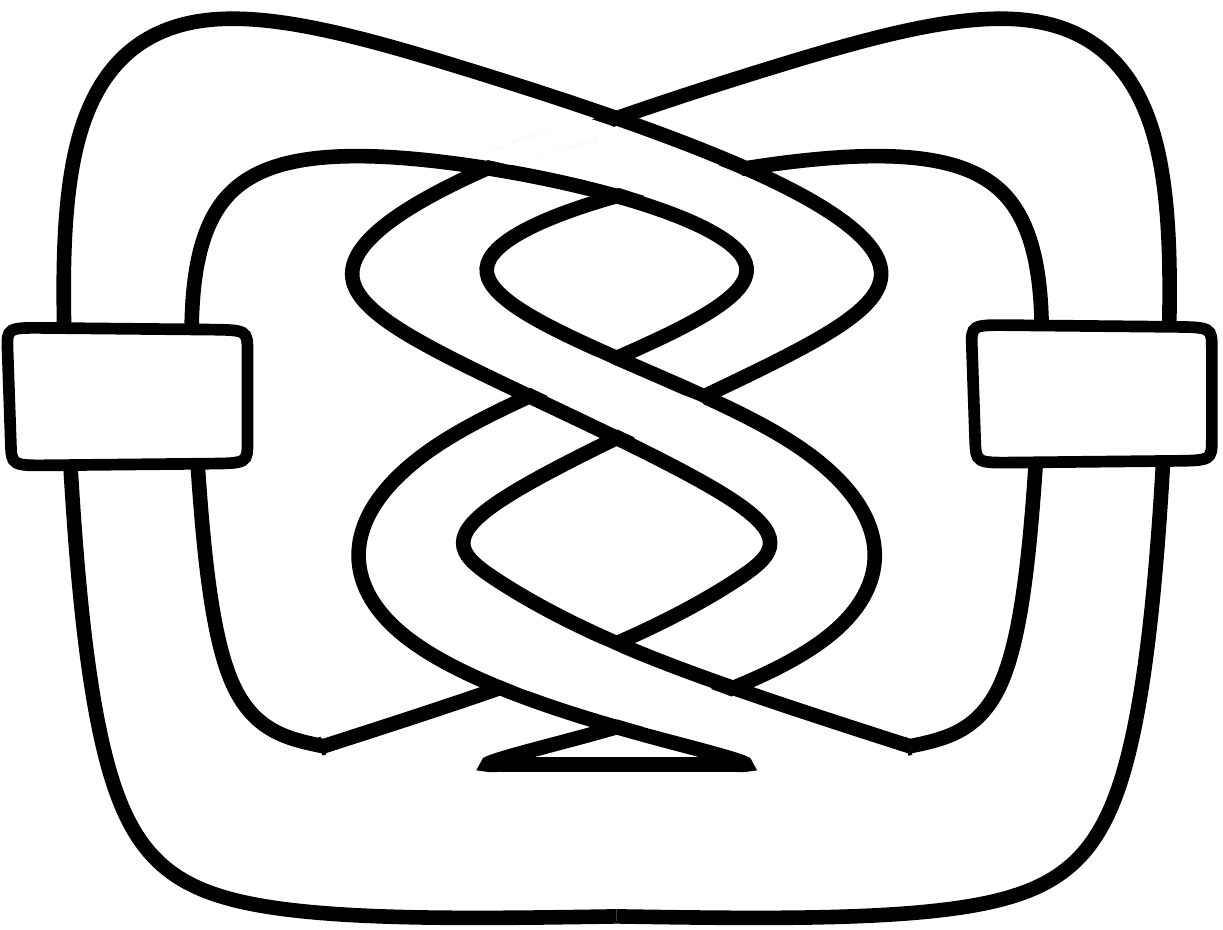}
\end{subfigure}
\caption{On the left: the knot $R:=9_{46}$ with the infections curves $\eta_1,\eta_2$; on the right: the satellite knot $K:=R(J_1,J_2)$ obtained by infecting $R$ along the curves $\eta_1,\eta_2$.}
\label{fig:946Satellite}
\end{figure}

Let $\eta_1$ and $\eta_2$ in $X_R$ be the curves shown on the left hand side of Figure~\ref{fig:946Satellite}.
Perform the satellite operation on $R$ along $\eta_1$ and $\eta_2$ with infection knots $J_1$ and $J_2$ respectively, to obtain a knot that we denote $K:= R(J_1,J_2)$ and that is depicted schematically on the right hand side of Figure~\ref{fig:946Satellite}.

The next theorem requires the existence of two hyperbolic Alexander polynomial one knots~$J_1$ and $J_2$ with exteriors that are not homeomorphic.  This is guaranteed by~\cite[Theorem~1.1]{Friedl-hyperbolic} applied to a Seifert matrix for the unknot.

\begin{theorem}\label{thm:without-rel-bdy}
Let $J_1$ and $J_2$ be two hyperbolic Alexander polynomial one knots with exteriors that are not homeomorphic. The knot $K$ shown on the right hand side of Figure~\ref{fig:946Satellite} has precisely two $G$-homotopy ribbon discs up to ambient isotopy.
\end{theorem}

\begin{proof}
First, we may construct a $G$-homotopy ribbon disc $D_1$ for $K$ by cutting the left hand band via a saddle move, to obtain the $(2,0)$ cable of $J_2$, and then capping this off with two parallel copies of the $\Z$-homotopy ribbon disc for $J_2$ whose existence is guaranteed by the~$\Delta_{J_2}(t) = 1$ condition. That this is a $G$-homotopy ribbon disc follows from the same calculation as in Section~\ref{sec:Examples}: two parallel copies of the $\Z$-homotopy ribbon disc for $J_2$ in $D^4$ have complement with fundamental group free of rank two generated by the meridians to the two components, just like the standard slice discs for the unlink given by the dotted circles in Figure~\ref{fig:Kirby}.

Construct a similar $G$-homotopy ribbon disc $D_2$ for $K$ by cutting the right hand band. There are still only two lagrangians for the Blanchfield form, so there are still only at most two $G$-homotopy ribbon discs up to ambient isotopy by Theorem~\ref{thm:Intro}.  To complete the proof of  Theorem~\ref{thm:without-rel-bdy} we need to argue that there is no isotopy of $K$ interchanging the two lagrangians.  If there were such an isotopy, then it would induce a self-homeomorphism $F \colon X_K \to X_K$ interchanging the classes of $\eta_1,\eta_2 \in H_1(X_K;\Z[t^{\pm 1}])$.

Recall the Jaco-Shalen-Johannson (JSJ) theorem~\cite[Theorem~1.9]{Hatcher-3manifolds}: \emph{let $M$ be a compact, irreducible, orientable 3-manifold. There is a collection $T$ of disjoint incompressible tori such that each component of $M$ cut along $T$ is either atoroidal (every incompressible torus is boundary parallel) or a Seifert manifold.  A minimal collection of such $T$ is unique up to isotopy.}

The knot exterior $X_K$ is certainly compact, orientable, and irreducible.  We need to identify the JSJ tori: they correspond to the satellite construction.

\begin{claim}
  The JSJ pieces of the knot exterior $X_K$ are $X_{R,\eta} := X_{R} \setminus (\nu \eta_1 \cup \nu \eta_2)$  together with the knot exteriors $X_{J_1}$ and $X_{J_2}$. The JSJ tori are $T_i := \partial \overline{\nu \eta_i}$, $i=1,2$.
\end{claim}

 \begin{proof}
 To prove the claim, first we argue that the tori $T_i$ are incompressible.
To see this, note that the longitude of $T_i$ is a generator of the Alexander module of $R$, therefore is nontrivial in $\pi_1(X_{R})$, so also in $\pi_1(X_{R,\eta})$.
The meridian of $T_i$ is a longitude in $X_{J_i}$, so is nontrivial in~$\pi_1(X_{J_i})$ by the loop theorem and the fact that $J_i$ is knotted.

 Next, both $J_1$ and $J_2$ are hyperbolic knots, so $X_{J_1}$ and $X_{J_2}$ are atoroidal.
  Similarly, using SnapPy, we checked that the link $R \cup \eta_1 \cup \eta_2$ is hyperbolic, and so $X_{R,\eta}$ cannot be decomposed further along tori.
   This completes the proof of the claim on the JSJ decomposition of $X_K$.
\end{proof}

 Now we show that there is no isotopy of $K$ interchanging the two lagrangians.
If there were, there would be a self-homeomorphism of $X_K$ with the same effect.
By the JSJ theorem it would have to switch the two JSJ tori, up to an isotopy of the self-homeomorphism.
Note that a longitude of the torus $\partial \overline{\nu \eta_i}$ generates the lagrangian $P_i$, for $i=1,2$.
But the JSJ pieces $X_{J_1}$ and $X_{J_2}$ are not homeomorphic, so the tori $\partial \overline{\nu \eta_i}$ and cannot be exchanged by any homeomorphism.  Therefore the two slice discs $D_1$ and $D_2$ are not ambiently isotopic.
\end{proof}

\bibliography{biblioSliceDiscs}
\bibliographystyle{alpha}

\end{document}